\documentclass[reqno]{amsart}

\usepackage{graphicx,enumerate,booktabs,nicefrac,color,bm}
\usepackage{stmaryrd}
\usepackage{amsmath, amssymb, amstext, amsfonts, mathrsfs}
\newcommand{\norm}[1]{\left\|#1\right\|}

\newcommand{\N}{\mathbb{N}}

\newcommand{\abs}[1]{\left|#1\right|}
\newcommand{\F}{\mathsf{F}}
\newcommand{\NF}{\mathsf{N}_{\F}}
\renewcommand{\d}{\mathsf{d}}
\newcommand{\ds}{\,\d s}
\newcommand{\dx}{\,\d\bm x}

\newcommand{\I}{\mathsf{I}_h}
\newcommand{\jmp}[1]{\left\llbracket#1\right\rrbracket}
\newcommand{\NN}[1]{\left|\!\left|\!\left|#1\right|\!\right|\!\right|}
\newcommand{\sign}{\mathrm{sign}}

\newtheorem{theorem}{Theorem}[section]
\newtheorem{lemma}[theorem]{Lemma}
\newtheorem{proposition}[theorem]{Proposition} 
\newtheorem{cor}[theorem]{Corollary}

\theoremstyle{definition}
\newtheorem{example}[theorem]{Example}

\newtheorem{algorithm}[theorem]{Algorithm}
\newtheorem{remark}[theorem]{Remark}

\title[Fully Adaptive Newton-Galerkin Methods]{Fully Adaptive Newton-Galerkin Methods for Semilinear Elliptic Partial Differential Equations}

\author[M.~Amrein]{Mario Amrein}
\author[T.~P.~Wihler]{Thomas P.~Wihler}
\address{Mathematics Institute, University of Bern, CH-3012 Switzerland}
\email{mario.amrein@math.unibe.ch}
\email{wihler@math.unibe.ch}

\begin{document}

\begin{abstract}
In this paper we develop an adaptive procedure for the numerical solution of general, semilinear elliptic problems with possible singular perturbations. Our approach combines both a prediction-type adaptive Newton method and an adaptive finite element discretization (based on a robust {\em a posteriori} error analysis), thereby leading to a fully adaptive Newton-Galerkin scheme. Numerical experiments underline the robustness and reliability of the proposed approach for different examples.
\end{abstract}



\maketitle

\section{Introduction}

The focus of this paper is the numerical approximation of semilinear elliptic problems with possible singular perturbations. More precisely, for a fixed parameter~$\varepsilon>0$ (possibly with~$\varepsilon\ll 1$), and a continuously differentiable function $f:\,\mathbb{R}\to\mathbb{R}$, we consider the problem of finding a function~$u:\,\Omega\to\mathbb{R}$ which satisfies
\begin{equation}\label{poisson}
\begin{aligned}
-\varepsilon \Delta u &=f(u) \text{ in } \ \Omega,\qquad
u=0 \text{ on } \partial \Omega. 
\end{aligned}
\end{equation}
Here, $\Omega\subset\mathbb{R}^d$, with $d=1$ or $d=2$, is an open and bounded 1d interval or a 2d Lipschitz polygon, respectively. Problems of this type appear in a wide range of applications including, e.g., nonlinear reaction-diffusion in ecology and chemical models~\cite{CaCo03,Ed05,Fr08,Ni11,OkLe01}, economy~\cite{BaBu95}, or classical and quantum physics~\cite{BeLi83,St77}. 

From an analysis point of view, semilinear elliptic boundary value problems~\eqref{poisson} have been studied in detail by a number of authors over the last decades; we refer, e.g., to the monographs~\cite{AmMa06,Ra86,Sm94} and the references therein. In particular, solutions of~\eqref{poisson} are known to be typically not unique (even infinitely many solutions may exist), and, in the singularly perturbed case, to exhibit boundary layers, interior shocks, and (multiple) spikes. The existence of multiple solutions due to the nonlinearity of the problem and/or the appearance of singular effects constitute two challenging issues when solving problems of this type numerically; see, e.g.,\cite{RoStTo08,Verhulst}.

Nowadays the use of the Newton-Raphson method in dealing with nonlinear phenomena is standard. Indeed, this method is highly successful if initial guesses are chosen close enough to a solution and if the basins of attraction for different solutions are sufficiently well-behaved for the Newton iteration to stay within the same attractor. As a consequence, on a {\em local} level, the scheme is often celebrated for its quadratic convergence regime close to a root. From a {\em global} perspective, however, the Newton method is well-known to exhibit chaotic behavior. Indeed, applying the Newton method to algebraic systems of equations, for example, may result in highly complex or even fractal attractor boundaries of the associated roots; see, e.g., \cite{peitgen}. This is related to the fact that the Newton iteration may be unstable in the sense that, farther away from a root, iterates may switch from one basin of attraction to another, and hence, converge to an undesired root (or even diverge). In the context of semilinear elliptic PDE the situation is even worse (and yet more severe in the singularly perturbed case): In fact, for certain types of problems, the Newton iteration will typically tend to become unbounded, and hence, will not approach a sensible solution at all; see, e.g., \cite{ChNiZo00}, where this issue has been addressed for a certain class of problems by means of a suitable rescaling technique in each step. A frequently employed remedy to tame (although not to eliminate) the chaotic behavior of Newton's method is the use of damping to avoid the appearance of possibly large updates in the iterations. An even more sophisticated way to further improve the quality of the results is the application of {\em variable} damping; see, e.g., the extensive overview~\cite{5} or \cite{DSB95,epureanu:102} for different variations of the classical Newton scheme. The idea of adaptively adjusting the magnitude of the Newton updates has also been studied in the recent articles~\cite{AmWi14,ScWi11}; there, following, e.g.,~\cite{neuberger,peitgen,smale}, the Newton method was identified as the numerical discretization of a specific ordinary differential equation (ODE)---the so-called continuous Newton method---by the explicit Euler scheme, with a fixed step size~$k=1$. Then, in order to tame the chaotic behavior of the Newton iterations, the idea presented in~\cite{AmWi14,ScWi11} is based on discretizing the continuous Newton ODE by the explicit Euler method with variable step sizes, and to combine it with a simple step size control procedure; in particular, the resulting algorithm retains the optimal step size~$k=1$ whenever sensible and is able to deal with singularities in the iterations more carefully than the classical Newton scheme. In fact, numerical experiments for algebraic and differential equations in~\cite{AmWi14,ScWi11} revealed that the new method is able to generate attractors with almost smooth boundaries, whereas the traditional Newton method produces fractal Julia sets; moreover, the numerical tests demonstrated an improved convergence rate not matched on average by the classical Newton
method.

In the present paper, our goal is to extend the approach developed in~\cite{AmWi14,ScWi11} to the numerical solution of~\eqref{poisson}. To this end, we will start by applying an adaptive Newton scheme, which is based on some simple prediction strategies, to the {\em nonlinear} boundary value problem~\eqref{poisson}. Subsequently, we discretize the resulting sequence of {\em linear} problems by a standard $\mathbb{P}_1$-finite element method (FEM); note that this approach is in contrast to solving the nonlinear algebraic system resulting from a FEM discretization of the original PDE with the aid of the Newton method (see, e.g., the work on inexact Newton methods~\cite{ErVo13}). In order to control the approximation error caused by the FEM discretization, we derive a residual-based {\em a posteriori} error analysis which allows to adaptively refine the finite element mesh; here, following the approach in~\cite{Verfuerth}, we will take particular care of the singular perturbation in order to obtain $\varepsilon$-robust error estimates. The final error estimate (Theorem~\ref{thm:1}) bounds the error in terms of the (elementwise) finite element approximation (FEM-error) and the error caused by the linearization of the original problem due to Newton's method (Newton-error). Then, in order to define a fully adaptive Newton-Galerkin scheme, we propose an interplay between the adaptive Newton-Raphson method and the adaptive finite element approach: More precisely, as the adaptive procedure is running, we either perform a Newton-Raphson step in accordance with our prediction strategy (Section~\ref{sc:anewton}) or refine the current mesh based on the {\em a posteriori error} analysis (Section~\ref{sc:FEM}), depending on which error (FEM-error or Newton-error) is  more dominant in the current iteration step. Our numerical results will reveal that sensible solutions can be found even in the singularly perturbed case, and that our scheme is reliable for reasonable choices of initial guesses, and $\varepsilon$-robust.

For the purpose of this paper, we suppose that a (not necessarily unique) solution~$u\in X:=H^1_0(\Omega)$ of~\eqref{poisson} exists; here, we denote by $H^1_0(\Omega)$ the standard Sobolev space of functions in~$H^1(\Omega)=W^{1,2}(\Omega)$ with zero trace on~$\partial\Omega$. Furthermore, signifying by~$X'=H^{-1}(\Omega)$ the dual space of~$X$, and upon defining the map $\F_{\varepsilon}: X\rightarrow X'$ through
\begin{equation}\label{eq:Fweak}
% \label{goal}
\left \langle \F_{\varepsilon}(u),v\right \rangle :=  \int_{\Omega}\left\{\varepsilon \nabla u\cdot \nabla v-f(u)v\right\}\dx\qquad \forall v\in X,
\end{equation}
where $\left\langle\cdot,\cdot\right\rangle$ is the dual product in~$X'\times X$, the above problem~\eqref{poisson} can be written as a nonlinear operator equation in~$X'$:
\begin{equation}\label{eq:F0}
u\in X:\qquad \F_\varepsilon(u)=0.
\end{equation}
In addition, on any subset~$D\subseteq\Omega$, we introduce the norm
\begin{equation}
\label{eq:Norm}
\NN{u}_{\varepsilon,D}:=\Bigl(\varepsilon\norm{\nabla u}_{0,D}^2 +\norm{u}_{0,D}^2 \Bigr)^{\nicefrac{1}{2}},
\end{equation}
where~$\|\cdot\|_{0,D}$ denotes the $L^2$-norm on~$D$. Note that, in the case of~$f(u)=-u$, when~\eqref{poisson} is linear and strongly elliptic, the norm $\NN{\cdot}_{\varepsilon,\Omega}$ is a natural energy norm on~$X$. Frequently, for~$D=\Omega$, the subindex~`$D$' will be omitted.  Furthermore, the associated dual norm of~$\F_\varepsilon$ from~\eqref{eq:Fweak}  is given by
\[
\NN{\F_\varepsilon(u)}_{X',\varepsilon} =  \sup_{\genfrac{}{}{0pt}{}{v\in X}{\NN{v}_{\varepsilon}=1}}\int_{\Omega}\left\{\varepsilon \nabla u\cdot \nabla v-f(u)v\right\}\dx.
\]
Throughout this work we shall use the abbreviation $x \preccurlyeq y$ to mean $x \leq c y$, for a constant $c>0$ independent of the mesh size~$h$ and of~$\varepsilon>0 $. 

The paper is organized as follows: In Section~\ref{sc:anewton} we will consider the Newton-Raphson method within the context of dynamical systems in general Banach spaces, and present two prediction strategies for controlling the Newton step size parameter. Furthermore, Section~\ref{sc:SEP} focuses on the application of the Newton-Raphson method to semilinear elliptic problems. In addition, we discuss the discretization of the problems under consideration by finite element methods in Section~\ref{sc:FEM}, and derive an $\varepsilon$-robust {\em a posteriori} error analysis. A series of numerical experiments illustrating the performance of the fully adaptive Newton-Galerkin scheme proposed in this work will be presented as well. Finally, we summarize our findings in Section~\ref{sc:concl}.

\section{Adaptive Newton-Raphson Methods in Banach Spaces}\label{sc:anewton}

In the following section we shall briefly revisit the adaptive Newton algorithm from~\cite{AmWi14}, and additionally, will derive an improved variant of our previous work.

\subsection{Abstract Framework}
Let $ X, Y $ be two Banach spaces, with norms~$\|\cdot\|_X$ and~$\|\cdot\|_Y$, respectively. Given an open subset~$\Xi\subset X$, and a (possibly nonlinear) operator~$\F:\,\Xi\to Y$, we consider the {\em nonlinear} operator equation
\begin{equation}\label{eq:F=0}
\F(u)=0,
\end{equation} 
for some unknown zeros~$u\in\Xi$. Supposing that the Fr\'echet derivative~$\F'$ of~$\F$ exists in~$\Xi$ (or in a suitable subset), the classical Newton-Raphson method for solving~\eqref{eq:F=0} starts from an initial guess~$u_0\in\Xi$, and generates the iterates
\begin{equation}\label{eq:newton}
u_{n+1}=u_n+\delta_n,\qquad n\ge 0,
\end{equation}
where the update~$\delta_n\in X$ is implicitly given by the {\em linear} equation
\[
\F'(u_n)\delta_n=-\F(u_n),\qquad n\ge 0.
\]
Naturally, for this iteration to be well-defined, we need to assume that~$\F'(u_n)$ is invertible for all~$n\ge 0$, and that~$\{u_n\}_{n\ge 0}\subset\Xi$. 

\subsection{A Simple Prediction Strategy}\label{sc:simple}

In order to improve the reliability of the Newton method~\eqref{eq:newton} in the case that the initial guess~$u_0$ is relatively far away from a root $u_{\infty}\in\Xi$ of~$\F$, $\F(u_\infty)=0$, introducing some damping in the Newton-Raphson method is a well-known remedy. In that case~\eqref{eq:newton} is rewritten as
\begin{equation}\label{eq:damped}
u_{n+1}=u_n-k_n\F'(u_n)^{-1}\F(u_n),\qquad n\ge 0,
\end{equation} 
where $k_n>0$, $n\ge 0$, is a damping parameter that may be adjusted {\em adaptively} in each iteration step. 

Provided that~$\F'(u)$ is invertible on a suitable subset of~$\Xi\subset X$, we define the {\em Newton-Raphson Transform} by
\[
u\mapsto\NF(u):=-\F'(u)^{-1}\F(u).
\]
Then, rearranging terms in~\eqref{eq:damped}, we notice that
\begin{equation}\label{eq:disc}
\frac{u_{n+1}-u_{n}}{k_n}=\NF(u_n), \qquad n\ge 0,
\end{equation}
i.e., \eqref{eq:damped} can be seen as the discretization of the {\em Davydenko-type system},
\begin{equation}\label{eq:davy}
\begin{split}
\dot{u}(t)&=\NF(u(t)), \quad t\geq 0,\qquad
 u(0)=u_0,
 \end{split}
\end{equation}
by the forward Euler scheme with step size~$k_n>0$. 

For~$t\in[0,\infty)$, the solution~$u(t)$ of~\eqref{eq:davy} defines a trajectory in~$X$ that begins at~$u_0$, and that will potentially converge to a zero of~$\F$ as~$t\to\infty$. Indeed, this can be seen (formally) from the integral form of~\eqref{eq:davy}, that is,
\begin{equation}\label{eq:int}
\F(u(t))=\F(u_0)e^{-t},\qquad t\ge 0,
\end{equation}
which implies that~$\F(u(t))\to 0$ as~$t\to\infty$.

Now taking the view of dynamical systems, our goal is to compute an upper bound for the value of the step sizes~$k_n>0$ from~\eqref{eq:damped}, $n\ge 0$, so that the discrete forward Euler solution~$\{u_n\}_{n\ge 0}$ from~\eqref{eq:damped} stays reasonably close to the continuous solution of~\eqref{eq:davy}. To this end, we approximate the trajectory~$u$ from~\eqref{eq:davy} close to the initial value~$u_0$ by a second-order Taylor expansion:
\begin{equation}\label{eq:u2}
u(t)\approx u_0+t\dot{u}(0)+t^2\xi,
\end{equation}
for some (fixed)~$\xi\in X$ to be determined. Using the integral form~\eqref{eq:int}, we see that
\[
\F(u_0)e^{-t}=\F(u(t))\approx \F(u_0+t\dot{u}(0)+t^2\xi),
\]
where a Taylor expansion of~$\F$ leads to $\F(u_0)e^{-t}\approx \F(u_0)+\F'(u_0)(t\dot{u}(0)+t^2\xi)$.
Moreover, from~\eqref{eq:davy} we observe that
\begin{equation}\label{eq:u0dot}
\dot{u}(0)=\NF(u_0),
\end{equation}
or equivalently, $\F'(u_0)\dot{u}(0)=-\F(u_0)$, and hence
$\F(u_0)(e^{-t}+t-1)\approx t^2\F'(u_0)\xi$.
Approximating~$e^{-t}\approx 1-t+\frac{1}{2}t^2$ results in
\begin{equation}\label{eq:xi}
\xi\approx \frac12\F'(u_0)^{-1}\F(u_0)=-\frac12\NF(u_0).
\end{equation}
Combining~\eqref{eq:u2} and \eqref{eq:xi} yields
\[
\|u(k_0)-u_1\|_X\approx k_0^2\|\xi\|_X\approx \frac{k_0^2}{2}\|\NF(u_0)\|_X,
\]
where~$u_1=u_0+k_0\NF(u_0)$ is the first Newton iterate from~\eqref{eq:damped} (with~$n=0$). Recalling that~$u_1$ may also be seen as the forward Euler approximation (with step size~$k_0>0$) of the solution~$u$ of~\eqref{eq:davy} at~$t=k_0$, the above relation can be understood as the nodal error between the solution of~\eqref{eq:davy} and its numerical approximation after the first time step. Then, for a given error tolerance~$\tau>0$, choosing
\[
k_0=\sqrt{\frac{2\tau}{\|\NF(u_0)\|_X}},
\]
we arrive at $\|u(k_0)-u_1\|_X\approx\tau$,
i.e., the exact trajectory given by the solution of~\eqref{eq:davy} and its forward Euler approximation from~\eqref{eq:damped} remain~$\tau$-close in the~$\|.\|_X$-norm for the given time step~$k_0$.

Iterating the above observations leads to the following prediction strategy for the selection of~$k_n$ in~\eqref{eq:damped}. Incidentally, the resulting algorithm is identical with the one presented in~\cite[Algorithm~2.1]{AmWi14} although our derivation here is different.

\begin{algorithm}~\label{al:simple}
Fix a tolerance $ \tau>0 $.
\begin{enumerate}[(i)]
\item  Start the Newton iteration with an initial guess $ u_{0}\in\Xi$.
\item In each iteration step $ n=0,1,2,\ldots $, compute
\begin{equation}
\label{eq:k1}
k_{n}=\min\left(\sqrt{\frac{2\tau}{\norm{\NF(u_{n})}_{X}}},1\right).
\end{equation}
\item Compute~$u_{n+1}$ based on the Newton iteration~\eqref{eq:damped}, and go to (ii) with $ n\leftarrow n+1 $. 
\end{enumerate}
\end{algorithm}

\begin{remark}
The minimum in~\eqref{eq:k1} ensures that the step size~$k_n$ is chosen to be~1 whenever possible. Indeed, this is required in order to guarantee quadratic convergence of the Newton iteration close to a root (provided that the root is simple).
\end{remark}

\begin{remark}\label{rm:norm}
Under certain conditions it can been proved that the above algorithm does in fact converge to a zero of~$\F$; see~\cite[Theorem~2.4]{AmWi14}.
\end{remark}

\subsection{An Improved Prediction Strategy} 

In Section~\ref{sc:simple} our step size prediction strategy is based on approximating the solution of the Davydenko-type system~\eqref{eq:davy} by the use of~\eqref{eq:u2}. We can improve this approach by looking at the Taylor expansion
\begin{equation}\label{eq:u3}
u(t)=u_0+t\dot{u}(0)+\frac12t^2\ddot{u}(0)+\mathcal{O}(t^3)
\end{equation}
of the trajectory~$u$ defined by~\eqref{eq:davy}. Recalling~\eqref{eq:u0dot} we can replace~$\dot{u}(0)$ above by the Newton-Raphson transform~$\NF(u_0)$, however, we still need to find a good approximation of~$\ddot{u}(0)$. This can be accomplished by taking the derivative of~\eqref{eq:davy} with respect to~$t$ at~$t=0$. Applying the chain rule gives
\[
\ddot{u}(0)=\NF'(u_0)\dot{u}(0)=\NF'(u_0)\NF(u_0).
\]
Since it is preferable to avoid the explicit appearance of~$\NF'(u_0)$ we look at, for some small~$h>0$, the Taylor expansion
\[
\NF(u_0+h\NF(u_0))=\NF(u_0)+h\NF'(u_0)\NF(u_0)+\mathcal{O}(h^2\|\NF(u_0)\|_X^2).
\]
We conclude
\[
\ddot{u}(0)=\NF'(u_0)\NF(u_0)=\frac{1}{h}\eta_h+\mathcal{O}(h\|\NF(u_0)\|_X^2),
\]
with~$\eta_h=\NF(u_0+h\NF(u_0))-\NF(u_0)$. Inserting this identity into~\eqref{eq:u3} and employing~\eqref{eq:u0dot}, we arrive at
\[
u(t)=u_0+t\NF(u_0)+\frac{t^2}{2h}\eta_h+\mathcal{O}(t^3)+\mathcal{O}(t^2h\|\NF(u_0)\|_X^2).
\]
Hence, after the first time step of length~$k_0>0$ there holds
\begin{equation}\label{eq:O}
u(k_0)-u_1=\frac{k_0^2}{2h}\eta_h+\mathcal{O}(k_0^3)+\mathcal{O}(k_0^2h\|\NF(u_0)\|_X^2),
\end{equation}
where~$u_1$ is the forward Euler solution from~\eqref{eq:damped}. Then, for a prescribed tolerance~$\tau>0$ as before, we have~$\|u(k_0)-u_1\|_X\approx\tau$ if we set $k_0=\sqrt{2\tau h\|\eta_h\|_X^{-1}}$.
In order to balance the~$\mathcal{O}$-terms in~\eqref{eq:O} it is reasonable to make the choice~$h=\mathcal{O}(k_0\|\NF(u_0)\|_X^{-2})$, i.e.,
\begin{equation}\label{eq:h}
h=\gamma k_0\|\NF(u_0)\|_X^{-2},
\end{equation}
for some parameter~$\gamma>0$. 

With these calculations we can improve the previous Algorithm~\ref{al:simple} as follows:

\begin{algorithm}~\label{al:zs}
Fix a tolerance $ \tau>0$ and a parameter~$\gamma>0$, and set~$n=0$.
\begin{enumerate}[(i)]
\item Start the Newton iteration with an initial guess $ u_{0}\in\Xi$.
\item If $n=0$, then choose
\[
\kappa_0=\min\left(\sqrt{\frac{2\tau}{\norm{\NF(u_{0})}_{X}}},1\right),
\]
according to Algorithm~\ref{al:simple}, else if~$n\ge 1$, let
$\kappa_{n}=k_{n-1}$.
Moreover, set $h_n=\gamma \kappa_n\|\NF(u_n)\|_X^{-2}$ based on~\eqref{eq:h}, and define
\begin{equation}\label{eq:knew}
k_{n}=\min\left(\sqrt{\frac{2\tau h_n}{\norm{\NF(u_0+h_n\NF(u_0))-\NF(u_0)}_{X}}},1\right).
\end{equation}
\item Compute~$u_{n+1}$ based on the Newton iteration~\eqref{eq:damped}, and go to~(ii) with $n\leftarrow n+1 $. 
\end{enumerate}
\end{algorithm}
\begin{remark}
In contrast to the simple prediction strategy from Section~\ref{sc:simple}, Algorithm~\ref{al:zs} is based on the improved Taylor approximation~\eqref{eq:u3}. This will naturally lead to more reliable results in the adaptive Newton iteration, since the discrete system~\eqref{eq:disc} will supposedly follow the continuous dynamics of~\eqref{eq:davy} more closely. Evidently, the price to pay is one additional evaluation of the Newton-Raphson transform in each time step of the discrete dynamical system~\eqref{eq:damped}; cf.~\eqref{eq:knew}. This will roughly increase the complexity of Algorithm~\ref{al:simple} by a constant factor of~2.
\end{remark}

\begin{remark}
The preset tolerance~$ \tau $ in the above adaptive strategies will typically be fixed {\em a priori}. Here,  for highly nonlinear problems featuring numerous or even infinitely many solutions, it is recommendable to select~$\tau\ll 1$ small in order to increase the chances of remaining within the attractor of the given initial guess. This is particularly important if the starting value is relatively far away from a solution.
\end{remark}

\section{Application to Semilinear Elliptic Problems}\label{sc:SEP}

In order to apply an adaptive Newton-Raphson method as introduced in Section~\ref{sc:anewton} to the nonlinear PDE problem~\eqref{eq:F0}, note that the Fr\'echet-derivative of $ \F_{\varepsilon} $ from~\eqref{eq:F0} at~$u\in X$ is given, by
\begin{equation*}
\left \langle \F_{\varepsilon}'(u)w,v\right \rangle = \int_{\Omega}\left\{\varepsilon \nabla w\cdot \nabla v-f'(u)wv\right\}\dx,\qquad v,w\in X=H^1_0(\Omega).
\end{equation*}
We note that, if there is a constant~$\beta>0$ for which~$f'(u)\in L^{1+\beta(d-1)}(\Omega)$, then $\F_{\varepsilon}'(u)$ is a well-defined linear mapping from~$X$ to~$X'$; see~Lemma~\ref{lm:app}.

Then, given an initial guess~$u_0\in X$ for~\eqref{eq:F0}, the Newton method~\eqref{eq:damped} is to find~$u_{n+1}\in X$ from~$u_n\in X$, $n\ge 0$, such that
\begin{equation*}
\F_\varepsilon'(u_n)(u_{n+1}-u_n)=-k_n\F_\varepsilon(u_n),
\end{equation*}
in~$X'$. 
Equivalently,
\begin{equation}\label{eq:weak}
a_\varepsilon(u_n; u_{n+1},v)=a_\varepsilon(u_n; u_{n},v)-k_n\ell_\varepsilon(u_n;v)\qquad\forall v\in X,
\end{equation}
where, for {\em fixed}~$u\in X$, 
\begin{equation*}
\begin{aligned}
a_{\varepsilon}(u;w,v)&:=\int_{\Omega}\left\{\varepsilon \nabla w\cdot \nabla v-f'(u)wv\right\}\dx ,\\
l_{\varepsilon}(u;v)&:=\int_{\Omega}\left\{\varepsilon \nabla u\cdot \nabla v-f(u)v\right\}\dx
\end{aligned}
\end{equation*}
are bilinear and linear forms on~$X\times X$ and~$X$, respectively.

\begin{remark} Let us consider a special case, where the weak formulation~\eqref{eq:weak}, for given~$u_n$, always has a (unique) solution $u_{n+1} \in X $. To this end, we assume that there are constants~$\underline{\lambda},\overline{\lambda}\ge 0$ with $\varepsilon C_P^{-2}>\overline{\lambda}$ such that $ -\underline\lambda \leq f'(u) \leq \overline{\lambda}$ holds for all $u\in\mathbb{R} $. Here, $C_P=C_P(\Omega)>0$ is the constant in the Poincar\'e inequality on~$\Omega$:
\begin{equation}\label{eq:P}
\|w\|_{0}\le C_P\|\nabla w\|_{0}\qquad\forall w \in X.
\end{equation}
Then, for any given $ u_n \in X $ the linear problem~\eqref{eq:weak} has a unique solution~$u_{n+1}\in X$.
\end{remark}

\begin{proof}
Our goal is to apply the Lax-Milgram Lemma. For this purpose, we will show that~$a_\varepsilon(u_n;\cdot,\cdot)$ is a bounded and coercive bilinear form on~$X\times X$, and that $l_{\varepsilon}(u_n;\cdot)$ is a bounded linear form on~$X$.

By definition of the bilinear form~$a_\varepsilon(u_n;\cdot,\cdot)$ we have
\begin{align*}
a(u_n;w,w)&=\int_{\Omega}\{\varepsilon |\nabla w|^2-f'(u_n)w^2\}\dx\\
&=\varepsilon\int_{\Omega}{ |\nabla w|^2}\dx-\int_{M_f^\complement}{f'(u_n)w^2}\dx-\int_{M_f}{f'(u_n)w^2} \dx.
\end{align*}
Here, $M_f=\{\bm x\in\Omega:\,(f'\circ u_n)(\bm x)>0\}$.
Then,
\begin{equation}\label{eq:an1}
a(u_n;w,w)\geq  \varepsilon\int_{\Omega}{ |\nabla w|^2}\dx-\overline{\lambda}\int_{M_f}{w^2} \dx
\geq  \varepsilon\int_{\Omega}{ |\nabla w|^2}\dx-\overline{\lambda}\int_{\Omega}{w^2} \dx.
\end{equation}
Invoking the Poincar\'e inequality~\eqref{eq:an1} results in
$a(u_n;w,w)\geq (\varepsilon- \overline{\lambda}C_{p}^{2})\norm{\nabla w}_{0,\Omega}^{2}$,
which, by the equivalence of the~$H^1$-seminorm and the norm~$\NN{\cdot}_\varepsilon$ from~\eqref{eq:Norm} on~$X$ (resulting from the Poincar\'e inequality~\eqref{eq:P}), shows that~$a(u_n;\cdot,\cdot)$ is coercive by assumption on the difference $\varepsilon- \overline{\lambda}C_{p}^{2}>0$.

Furthermore, $a(u_n;\cdot,\cdot)$ is bounded. Indeed, for~$v,w\in X$ there holds
\begin{align*}
|a(u_n;v,w)|&\le \varepsilon \int_{\Omega}|\nabla w||\nabla v|\dx+\sup_{x\in\mathbb{R}}|f'(x)|\int_\Omega |w||v|\dx\\
&\le \varepsilon \int_{\Omega}|\nabla w||\nabla v|\dx+\max(\underline\lambda,\overline\lambda)\int_\Omega |w||v|\dx.
\end{align*}
Applying the Cauchy-Schwarz inequality, we obtain
\begin{equation*}
\begin{aligned}
|a(u_n;v,w)|
&\preccurlyeq\left(\varepsilon\|\nabla v\|^2_{0}+\|v\|^2_{0}\right)^{\nicefrac12}
\left(\varepsilon\|\nabla w\|^2_{0}+\|w\|^2_{0}\right)^{\nicefrac12}
=\NN{v}_{\varepsilon,\Omega}\NN{w}_{\varepsilon,\Omega},
\end{aligned}
\end{equation*}
which shows the continuity of $a(u_n;\cdot,\cdot) $.

Let us now focus on $\ell_\varepsilon(u_n;v)$:
For~$v\in X$, the Cauchy-Schwarz inequality yields
\begin{equation}{\label{bound2}}
|\ell_\varepsilon(u_n;v)|\le 
\left(\varepsilon\|\nabla u_n\|^2_{0}+\|f(u_n)\|^2_{0}\right)^{\nicefrac12}
\left(\varepsilon\|\nabla v\|^2_{0}+\|v\|^2_{0}\right)^{\nicefrac12}.
\end{equation}
Noting that by the Lipschitz continuity of $f$, there holds $\abs{f(u_n)}
\leq m\abs{u_n}+c$, with~$m=\max(\underline\lambda,\overline\lambda)$. Hence, we see that
\begin{align*}
\|f(u_n)\|^2_{0,\Omega} 
&\le \int_\Omega\left(m|u_n|+c\right)^2\dx
\le 2\int_\Omega\left(m^2|u_n|^{2}+c^2\right)\dx\\
&\le 2m^2\|u_n\|_{0}^{2}+2c^2|\Omega|
\preccurlyeq \NN{u_n}_{\varepsilon}^2+|\Omega|,
\end{align*}
for any~$u\in X$. Inserting into~\eqref{bound2} we end up with
\[
\norm{\ell_\varepsilon(u_n;\cdot)}_{X'}\preccurlyeq \left(\NN{u_n}_{\varepsilon,\Omega}^2+|\Omega|\right)^{\nicefrac12},
\]
i.e., the linear form~$v\mapsto \ell_\varepsilon(u_n;v)$ is bounded.

The above calculations show that, for any fixed~$u_n\in X$, the linear form $v\mapsto a_\varepsilon(u_n; u_n,v)-k_n\ell_\varepsilon(u_n;v)$ is bounded. Hence, recalling the coercivity and continuity of~$a_\varepsilon(u_n;\cdot,\cdot)$, the linear problem~\eqref{eq:weak} possesses a unique solution $u_{n+1} \in X$ by the Lax-Milgram Lemma.
\end{proof}

\section{Newton-Galerkin Finite Element Discretization}\label{sc:FEM}

In order to provide a numerical approximation of~\eqref{poisson}, we will discretize the weak formulation~\eqref{eq:weak} by means of a finite element method, which, in combination with the Newton-Raphson iteration, constitutes a Newton-Galerkin approximation scheme. Furthermore, we shall derive {\em a posteriori} error estimates for the finite element discretization which allow for an adaptive refinement of the meshes in each Newton step. This, together with the adaptive prediction strategies from Section~\ref{sc:anewton}, leads to a fully adaptive Newton-Galerkin discretization method for~\eqref{poisson}.

\subsection{Finite Element Meshes and Spaces}
   
Let $ \mathcal{T}_{h}=\{T\}_{T\in\mathcal{T}_h}$, be a regular and shape-regular mesh partition of $\Omega $ into disjoint open simplices, i.e., any~$T\in\mathcal{T}_h$ is an affine image of the (open) reference simplex~$\widehat T=\{\widehat x\in\mathbb{R}_+^d:\,\sum_{i=1}^d\widehat x_i<1\}$. By~$h_T=\mathrm{diam}(T)$ we signify the element diameter of~$T\in\mathcal{T}_h$, and by $h=\max_{T\in\mathcal{T}_h}h_T$ the mesh size. Furthermore, by $\mathcal{E}_{h}$ we denote the set of all interior mesh nodes for~$d=1$ and interior (open) edges for~$d=2$ in~$\mathcal{T}_{h}$. In addition, for~$T\in\mathcal{T}_h$, we let~$\mathcal{E}_h(T)=\{E\in\mathcal{E}_h:\,E\subset\partial T\}$. For~$E\in\mathcal{E}_h$, we let~$h_E$ be the mean of the lengths of the adjacent elements in 1d, and the length of~$E$ in~2d.

We consider the finite element space of continuous, piecewise linear functions on $\mathcal{T}_{h}$ with zero trace on~$\partial\Omega$ given by
\begin{equation*}
V_{0}^{h}:=\{\varphi\in H^1_0(\Omega):\,\varphi|_{T} \in \mathbb{P}_{1}(T) \, \forall T \in \mathcal{T}_{h}\},
\end{equation*}
where~$\mathbb{P}_1(T)$ is the standard space of all linear polynomial functions on~$T$. 

Moreover, for any function $  \varphi \in V^h $ and a given edge $ E \in \mathcal{E}_{h}=\partial T^\sharp\cap \partial T^\flat $ shared by two neighboring elements~$T^\sharp, T^\flat\in\mathcal{T}_h$, we denote by $\jmp{\varphi}_{E} $ the jump of $ \varphi $ across~$E$:
\[
\jmp{\varphi}_{E}(\bm x)=\lim_{t\to 0^+}\varphi(\bm x+t\bm n^\sharp)\bm n^\sharp+\lim_{t\to 0^+}\varphi(\bm x+t\bm n^\flat)\bm n^\flat \qquad \forall \bm x \in E.
\]
Here, $\bm n^\sharp$ and~$\bm n^\flat$ denote the unit outward vectors on~$\partial T^\sharp$ and~$\partial T^\flat$, respectively. 

Furthermore, for~$T\in\mathcal{T}_h$, and~$E\in\mathcal{E}_h$, we set 
\[
 %w_{T}:=\bigcup_{\mathcal{E}_h(T)\cap\mathcal{E}_h(T')\neq\emptyset}{T'}, \qquad 
 \widetilde{w}_{T}:=\bigcup_{\genfrac{}{}{0pt}{}{T'\in\mathcal{T}_h:}{\overline T \cap \overline T'  \neq \emptyset}}{T'}, \qquad w_{E}:=\bigcup_{\genfrac{}{}{0pt}{}{T\in\mathcal{T}_h:}{E\subset\partial{T}}}{T}.
\]

\subsection{Approximation Results}

Let us recall the following classical quasi-interpo\-lation result.

\begin{proposition}\label{pr:clement}
Let $\I:\,H_{0}^{1}(\Omega)\rightarrow V_{0}^{h} $ be the quasi-interpolation Cl\'ement operator (see, e.g., \cite{Verfuerthbook}). Then, there holds the error estimate
\begin{equation*}
% \label{Clement}
\norm{\nabla^l (v-\I v)}_{0,T}\preccurlyeq \ h_{T}^{k-l}\norm{\nabla^k v}_{0,\widetilde{w}_{T}}
\end{equation*}
for all $ T\in \mathcal{T}_{h} $, all $l,k \in \N $ with $ 0\leq l\leq k\leq 1 $, and all $ v \in H^{k}(\widetilde{w}_{T}) $.
\end{proposition}

In order to provide $\varepsilon$-robust approximation results, we follow the approach presented in~\cite{Verfuerth} (see also~\cite{MeWi14}). More precisely, recalling Proposition~\ref{pr:clement}, we have 
\[
 \norm{v-\I v}_{0,T}^2 \preccurlyeq \norm{v}_{0,\widetilde{w}_T}^2, \qquad  \norm{v-\I v}_{0,T}^2 \preccurlyeq \varepsilon^{-1}h_T^2\varepsilon \norm{\nabla v}_{0,\widetilde{w}_T}^2,
\]
for any~$T\in\mathcal{T}_h$. Thus, if we set 
\begin{equation}\label{eq:alphaT}
\alpha_{T}:=\min(1,\varepsilon^{-\nicefrac{1}{2}}h_T),
\end{equation}
we find
\begin{equation}
 \label{estimate1}
 \norm{v-\I v}_{0,T}\preccurlyeq \alpha_{T}\NN{v}_{\varepsilon,\widetilde{w}_T}.
\end{equation}
Furthermore, recalling the well-known multiplicative trace inequality,
\[
\|w\|_{0,E}^2\preccurlyeq h_T^{-1}\|w\|_{0,T}^2+\|w\|_{0,T}\|\nabla w\|_{0,T}\qquad\forall w\in H^1(T),\,\forall E\in\mathcal{E}_h(T),
\]
for any~$T\in\mathcal{T}_h$, we have
\[
\norm{v-\I v}_{0,E}^2\preccurlyeq h_T^{-1}\|v-\I v\|_{0,T}^2+\|v-\I v\|_{0,T}\|\nabla(v-\I v)\|_{0,T},
\]
for any~$E\in\mathcal{E}_h$ with~$E\subset\partial T$. Inserting~\eqref{estimate1} and employing Proposition~\ref{pr:clement}, we arrive at
\begin{align*}
\norm{v-\I v}_{0,E}^2
&\preccurlyeq h_T^{-1}\alpha_T^2\NN{v}^2_{\varepsilon,\widetilde{w}_T}+\alpha_T\NN{v}_{\varepsilon,\widetilde{w}_T}\|\nabla v\|_{0,\widetilde{w}_T}\\
&\preccurlyeq h_T^{-1}\alpha_T^2\NN{v}^2_{\varepsilon,\widetilde{w}_T}+\varepsilon^{-\nicefrac12}\alpha_T\NN{v}_{\varepsilon,\widetilde{w}_T}^2\\
&\preccurlyeq \left(h_T^{-1}\min(1,\varepsilon^{-\nicefrac12}h_T)^2+\varepsilon^{-\nicefrac12}\min(1,\varepsilon^{-\nicefrac12}h_T)\right)\NN{v}_{\varepsilon,\widetilde{w}_T}^2\\
&\preccurlyeq \min(1,\varepsilon^{-\nicefrac12}h_T)\left(\min(h_T^{-1},\varepsilon^{-\nicefrac12})+\varepsilon^{-\nicefrac12}\right)\NN{v}_{\varepsilon,\widetilde{w}_T}^2\\
&\preccurlyeq \varepsilon^{-\nicefrac12}\min(1,\varepsilon^{-\nicefrac12}h_T)\NN{v}_{\varepsilon,\widetilde{w}_T}^2
\end{align*}
Hence,
\[
\norm{v-\I v}_{0,E}\preccurlyeq \varepsilon^{-\nicefrac14}\alpha_T^{\nicefrac12}\NN{v}_{\varepsilon,\widetilde{w}_T},
\]
and by shape-regularity of the mesh~$\mathcal{T}_h$,
\[
\norm{v-\I v}_{0,E}\preccurlyeq \varepsilon^{-\nicefrac14}\alpha_E^{\nicefrac12}\NN{v}_{\varepsilon,\widetilde{w}_T},
\]
with
\begin{equation}\label{boundary}
\alpha_E:=\min(1,\varepsilon^{-\nicefrac12}h_E),\qquad E\in\mathcal{E}_h.
\end{equation}

Let us summarize the above estimates:
\begin{cor}\label{co:appr}
Let $\I:\,H_{0}^{1}(\Omega)\rightarrow V_{0}^{h} $ be the quasi-interpolation Cl\'ement operator from Proposition~\ref{pr:clement}. Then, for any elements~$T, T'\in\mathcal{T}_h$, and an edge~$E=\mathcal{E}_h(T)\cap\mathcal{E}_h(T')$, and any~$v\in H^1_0(\Omega)$, there hold the approximation bounds
\begin{align*}
\norm{v-\I v}_{0,T}&\preccurlyeq \alpha_{T}\NN{v}_{\varepsilon,\widetilde{w}_T},\\
\norm{v-\I v}_{0,E}&\preccurlyeq \frac12\varepsilon^{-\nicefrac14}\alpha_E^{\nicefrac12}\left(\NN{v}_{\varepsilon,\widetilde{w}_T}+\NN{v}_{\varepsilon,\widetilde{w}_{T'}}\right),
\end{align*}
where~$\alpha_T$ and~$\alpha_E$ are defined in~\eqref{eq:alphaT} and~\eqref{boundary}, respectively.
\end{cor}

\subsection{Linear Finite Element Discretization}

We consider the finite element approximation of~\eqref{eq:weak} which is to find~$u_{n+1}^h\in V_{0}^h$ from a given $u_n\in X$, $n\ge 0$, (with~$u_0\in X$ being an initial guess) such that
\begin{equation}\label{eq:fem}
a_\varepsilon(u_n^h; u_{n+1}^h,v)=a_\varepsilon(u_n^h; u_{n}^h,v)-t\ell_\varepsilon(u_n^h;v)\qquad\forall v\in V_{0}^{h}.
\end{equation}
Here, $t$ takes the role of a parameter which corresponds to the step size in the adaptive Newton scheme. 
Introducing
\begin{equation}\label{eq:ut}
u_{n+1}^{(t,h)}:=u_{n+1}^h-(1-t)u_{n}^h
\end{equation}
and
\begin{equation}\label{eq:ft}
f^t(u_{n+1}^h):=tf(u_{n}^h)+f'(u_{n}^{h})(u_{n+1}^h-u_n^h),
\end{equation}
and rearranging terms, \eqref{eq:fem} can be rewritten as

\begin{equation}
 \label{start}
 \begin{aligned}
\varepsilon\int_{\Omega}{ \nabla u_{n+1}^{(t,h)} \cdot\nabla v} \dx& = \int_{\Omega}{ f^t (u_{n+1}^h) v }\dx \qquad \forall v  \in V_{0}^h.
\end{aligned}
\end{equation}

\subsection{{\em A Posteriori} Error Analysis}

The aim of this section is to derive {\em a posteriori} error bounds for~\eqref{start}.

\subsubsection{Upper Bound}
In order to measure the error between the finite element discretization~\eqref{eq:fem} and the original problem~\eqref{poisson}, a natural quantity to bound is the residual~$\F_{\varepsilon}(u_{n+1}^{h})$ in~$X'$. In order to proceed in this direction, we notice that the adaptively chosen damping parameter~$t$ in the Newton-Raphson method~\eqref{eq:fem} will equal~1 
sufficiently close to a root of~$\F_\varepsilon$. For this reason, we may focus on the `shifted' residual~$\F_{\varepsilon}(u_{n+1}^{(t,h)})$ in~$X'$ instead. 
To do so, let~$v\in H^1_0(\Omega)$. We begin with~\eqref{start}, which implies that
\[
\int_{\Omega}\left\{\varepsilon \nabla u_{n+1}^{(t,h)} \cdot\nabla \I v- f^t (u_{n+1}^h) \I v\right\}\dx=0,
\]
where~$\I v$ is the quasi-interpolant from Proposition~\ref{pr:clement}. Then,
\begin{align*}
\left \langle \F_{\varepsilon}(u_{n+1}^{(t,h)}),v \right \rangle 
&=\int_\Omega\left\{\varepsilon\nabla u_{n+1}^{(t,h)}\cdot\nabla v-f(u_{n+1}^{(t,h)})v\right\}\dx\\
&=\varepsilon\int_\Omega\nabla u_{n+1}^{(t,h)}\cdot\nabla (v-\I v)\dx
+\int_\Omega\left\{f^t (u_{n+1}^h) \I v-f(u_{n+1}^{(t,h)})v\right\}\dx\\
&=\varepsilon\int_\Omega\nabla u_{n+1}^{(t,h)}\cdot\nabla (v-\I v)\dx
+\int_\Omega f^t (u_{n+1}^h) (\I v-v)\dx\\
&\quad+\int_\Omega\left\{f^t (u_{n+1}^h)-f(u_{n+1}^{(t,h)})\right\}v\dx.
\end{align*}
Integrating by parts elementwise in the first term yields
\begin{align*}
\int_\Omega\nabla u_{n+1}^{(t,h)}\cdot\nabla (v-\I v)\dx
%&=\sum_{T\in\mathcal{T}_h}\int_T\nabla u_{n+1}^{(t,h)}\cdot\nabla (v-\I v)\dx\\
&=-\sum_{T\in\mathcal{T}_h}\int_T\Delta u_{n+1}^{(t,h)}(v-\I v)\dx\\
&\quad+\sum_{T\in\mathcal{T}_h}\int_{\partial T}(\nabla u_{n+1}^{(t,h)}\cdot\bm n_T)(v-\I v)\ds.
\end{align*}
An elementary calculation, recalling the fact that~$(v-\I v)|_{\partial\Omega}=0$, shows that
\[
\sum_{T\in\mathcal{T}_h}\int_{\partial T}(\nabla u_{n+1}^{(t,h)}\cdot\bm n_T)(v-\I v)\ds
=\sum_{E\in\mathcal{E}_h}\int_E\jmp{\nabla u_{n+1}^{(t,h)}}(v-\I v)\ds.
\]
Therefore, we have the following result:

\begin{proposition}
Given~$u_{n+1}^{(t,h)}$ and~$f^t(u_{n+1}^h)$ from~\eqref{eq:ut} and~\eqref{eq:ft}, respectively. Then, there holds the identity
\begin{equation}\label{eq:Fres}
\left \langle \F_{\varepsilon}(u_{n+1}^{(t,h)}),v \right \rangle
=\sum_{E\in\mathcal{E}_h}a_E+\sum_{T\in\mathcal{T}_h}(b_T+c_T),
\end{equation}
where
\begin{align}
a_{E}&:= \int_{E}\varepsilon\jmp{\nabla u_{n+1}^{(t,h)}}(v-\I v)\ds,\qquad
c_{T}:=\int_{T}{\left\{f^t(u_{n+1}^h)-f(u_{n+1}^{(t,h)})\right\}v}\dx,\nonumber\\
b_{T}&:=\int_{T}{\left\{f^t(u_{n+1}^h)+\varepsilon \Delta u_{n+1}^{(t,h)}\right\}(\I v-v)}\dx,\label{bounds}
\end{align}
with~$E\in\mathcal{E}_h$, $T\in\mathcal{T}_h$.
\end{proposition}

Now, for~$T\in\mathcal{T}_h$, defining
\begin{equation}
\label{Newtonerror}
\delta_{n,T}:=\norm{f^{t}(u_{n+1}^h)-f(u_{n+1}^{(t,h)})}_{0,T},
\end{equation}
and
\begin{equation}
\label{Femerror}
\eta_{n,T}^2:= \alpha_{T}^2 \norm{f^t(u_{n+1}^h)+\varepsilon \Delta u_{n+1}^{(t,h)}}_{0,T}^2+\frac{1}{2}\sum_{E\in \mathcal{E}_{h}(T)}{\varepsilon^{-\nicefrac12}\alpha_E\norm{\varepsilon \jmp{\nabla u_{n+1}^{(t,h)} }}_{0,E}^2},
\end{equation}
with~$\alpha_T$ and~$\alpha_E$ from~\eqref{eq:alphaT} and~\eqref{boundary}, respectively,
we are ready to prove an upper {\em a posteriori} bound on the (shifted) residual.

\begin{theorem}
\label{thm:1}
Consider~$u_{n+1}^{(t,h)}$ from~\eqref{eq:ut}. Then, there holds the upper bound:
\begin{equation}
\label{upperbound}
\NN{\F(u_{n+1}^{(t,h)})}_{X',\varepsilon}^2
\preccurlyeq \delta_{n,\Omega}^2+\sum_{T\in \mathcal{T}_h}{\eta_{n,T}^{2}},
\end{equation}
with~$\delta_{n,\Omega}$ and~$\eta_{T,n}$ from~\eqref{Newtonerror} and~\eqref{Femerror}, respectively.
\end{theorem}

\begin{proof}
First let~$E\in\mathcal{E}_h$ and~$T,T'\in\mathcal{T}_h$ with~$E=\mathcal{E}_h(T)\cap\mathcal{E}_H(T')$.
Then, $a_E$ from~\eqref{bounds} can be estimated using Corollary~\ref{co:appr} as follows:
\begin{align*}
|a_E|&\le \varepsilon\norm{\jmp{\nabla u_{n+1}^{(t,h)}}}_{0,E}\norm{v-\I v}_{0,E}\\
&\preccurlyeq\frac12\varepsilon^{\nicefrac34}\alpha_E^{\nicefrac12}\norm{\jmp{\nabla u_{n+1}^{(t,h)}}}_{0,E}\left(\NN{v}_{\varepsilon,\widetilde{w}_T}+\NN{v}_{\varepsilon,\widetilde{w}_{T'}}\right).
\end{align*}
Applying the Cauchy-Schwarz inequality leads to
\begin{align*}
\left|\sum_{E\in\mathcal{E}_h}a_E\right|
&\preccurlyeq\left(\frac12\sum_{E\in \mathcal{E}_{h}}{\varepsilon^{\nicefrac32}\alpha_{E} \norm{\jmp{\nabla u_{n+1}^{(t,h)}}}_{0,E}^2}\right)^{\nicefrac12}
\NN{v}_{\varepsilon,\Omega}.
\end{align*}
Furthermore, again using Corollary~\ref{co:appr}, we see that
\begin{align*}
\left|\sum_{T\in\mathcal{T}_h}b_T\right|
&\preccurlyeq \sum_{T\in\mathcal{T}_h}\alpha_T\norm{f^t(u_{n+1}^h)+\varepsilon \Delta u_{n+1}^{(t,h)}}_{0,T}\NN{v}_{\varepsilon,\widetilde{w}_T}\\
&\preccurlyeq\left(\sum_{T\in\mathcal{T}_h}\alpha_T^2\norm{f^t(u_{n+1}^h)+\varepsilon \Delta u_{n+1}^{(t,h)}}_{0,T}^2\right)^{\nicefrac12}\NN{v}_{\varepsilon,\Omega}.
\end{align*}
Similarly, there holds
\begin{align*}
\left|\sum_{T\in\mathcal{T}_h}c_T\right|&
\preccurlyeq\sum_{T\in\mathcal{T}_h} \norm{f^t(u_{n+1}^h)-f(u_{n+1}^{(t,h)})}_{0,T}\norm{v}_{0,T}\\
&\preccurlyeq\left(\sum_{T\in\mathcal{T}_h}\norm{f^t(u_{n+1}^h)-f(u_{n+1}^{(t,h)})}_{0,T}^2\right)^{\nicefrac12}\NN{v}_{\varepsilon,\Omega}.
\end{align*}

Now, applying the Cauchy-Schwarz inequality to~\eqref{eq:Fres} we see that
\[
\begin{aligned}
\abs{\left \langle \F_{\varepsilon}(u_{n+1}^{(t,h)}),v \right \rangle}
&\preccurlyeq\sum_{E\in\mathcal{E}_h}|a_E|+\sum_{T\in\mathcal{T}_h}|b_T|+\sum_{T\in\mathcal{T}_h}|c_T|\\
&\preccurlyeq \left(\delta_{n,\Omega}^2+\sum_{T\in \mathcal{T}_h}{\eta_{n,T}^{2}}\right)^{\nicefrac{1}{2}}\NN{v}_{\varepsilon,\Omega}.
\end{aligned}
\]
Dividing by~$\NN{v}_{\varepsilon,\Omega}$, and taking the supremum for all~$v\in H^1_0(\Omega)$, completes the proof.
\end{proof}

\begin{remark}\label{rm:4.5}
Under certain conditions on the nonlinearity~$f$ in~\eqref{poisson}, the right-hand side of \eqref{upperbound} is equivalent to~$\NN{u-u_{n+1}^{(t,h)}}_{\varepsilon,\Omega}$. To explain this, for $v,w\in X$, we notice that
\begin{align*}
\left \langle \F_{\varepsilon}(v)-\F_{\varepsilon}(w),v-w  \right \rangle &=\int_{\Omega}{\left\{\varepsilon |\nabla{(v-w)}|^2-(f(v)-f(w))(v-w)\right\}}\dx.
\end{align*}
Then, supposing that there exists a constant~$\lambda>-C_P^{-2}\varepsilon$, where~$C_P$ is the Poincar\'e constant from~\eqref{eq:P}, such that~$(f(x)-f(y))(x-y)\le-\lambda(x-y)^2$ for all~$x,y\in\mathbb{R}$, we conclude that
\begin{align*}
\langle \F_{\varepsilon}(v)-\F_{\varepsilon}(w),v-w  \rangle 
%&=\int_{\Omega}{\left\{\varepsilon |\nabla{(v-w)}|^2-(f(v)-f(w))(v-w)\right\}}\dx\\
&\geq \int_{\Omega}{\left\{\varepsilon |\nabla{(v-w)}|^2+\lambda (v-w)^2\right\}}\dx.
\end{align*}
From this, for~$\beta\ge 0$, it follows that
\begin{align*}
\langle \F_{\varepsilon}(v)&-\F_{\varepsilon}(w),v-w  \rangle \\
&\geq \beta\varepsilon \norm{\nabla{(v-w)}}_{0,\Omega}^{2}+(1-\beta)\varepsilon \norm{\nabla{(v-w)}}_{0,\Omega}^{2}+\lambda\norm{v-w}_{0,\Omega}^{2}\\
&\ge\beta C_P^{-2}\varepsilon \norm{v-w}_{0,\Omega}^{2}+(1-\beta)\varepsilon \norm{\nabla{(v-w)}}_{0,\Omega}^{2}+\min(1,\lambda) \norm{v-w}_{0,\Omega}^{2}\\
&\ge\min(1-\beta,\beta C_P^{-2}\varepsilon+\min(1,\lambda)) \NN{v-w}_{\varepsilon,\Omega}^2.
\end{align*}
Choosing~$\beta=(C_P^{-2}\varepsilon+1)^{-1}(1-\min(1,\lambda))$, it holds
\begin{equation}\label{eq:low}
\langle \F_{\varepsilon}(v)-\F_{\varepsilon}(w),v-w  \rangle
\ge \frac{C_P^{-2}\varepsilon+\min(1,\lambda)}{C_P^{-2}\varepsilon+1}\NN{v-w}_{\varepsilon,\Omega}^2.
\end{equation}
By assumption on~$\lambda$, the constant on the right-hand side in the above inequality is positive. Moreover, if there exists a constant~$L>0$ such that there holds the Lipschitz condition $|f(x)-f(y)|\le L|x-y|$ for all~$x,y\in\mathbb{R}$, then, 
for $ z \in X $, we observe that
\[
 \begin{aligned}
\abs{\left \langle \F_{\varepsilon}(v)-\F_{\varepsilon}(w),z \right \rangle}
&\leq \int_{\Omega}{\left\{\varepsilon \abs{\nabla{(v-w)}\cdot\nabla{z}}+\abs{(f(v)-f(w))z}\right\}}\dx\\
&\leq \int_{\Omega}{\left\{\varepsilon \abs{\nabla{(v-w)}}\abs{\nabla{z}}+L|v-w||z|\right\}}\dx.
\end{aligned}
\]
Using the Cauchy-Schwarz inequality, yields
\begin{align}
\abs{\left \langle \F_{\varepsilon}(v)-\F_{\varepsilon}(w),z \right \rangle}
&\leq \big(\varepsilon \norm{\nabla{(v-w)}}_{0}^{2}+L^2\norm{v-w}_{0}^2\big)^{\nicefrac12}\big(\varepsilon \norm{\nabla{z}}_{0}^{2}+\norm{z}_{0}^2\big)^{\nicefrac12}\nonumber\\
&\leq \max(1,L)\NN{v-w}_{\varepsilon,\Omega}\NN{z }_{\varepsilon,\Omega}.\label{eq:up}
\end{align}

Now, if~$u$ is the exact solution of~\eqref{poisson}, and~$u_{n+1}^{(t,h)}\neq u$ from~\eqref{eq:ut}, then~\eqref{eq:low} implies that
\begin{align*}
\langle\F_\varepsilon(u_{n+1}^{(t,h)}),u_{n+1}^{(t,h)}-u\rangle
&=\langle\F_\varepsilon(u_{n+1}^{(t,h)})-\F_\varepsilon(u),u_{n+1}^{(t,h)}-u\rangle\\
&\ge \frac{C_P^{-2}\varepsilon+\min(1,\lambda)}{C_P^{-2}\varepsilon+1}\NN{u-u_{n+1}^{(t,h)}}_{\varepsilon,\Omega}^2,
\end{align*}
and thus,
\begin{align*}
\NN{u-u_{n+1}^{(t,h)} }_{\varepsilon,\Omega}
&\le \frac{C_P^{-2}\varepsilon+1}{C_P^{-2}\varepsilon+\min(1,\lambda)}\frac{\langle\F_\varepsilon(u_{n+1}^{(t,h)}),u_{n+1}^{(t,h)}-u\rangle}{\NN{u-u_{n+1}^{(t,h)} }_{\varepsilon,\Omega}}\\
&\le \frac{C_P^{-2}\varepsilon+1}{C_P^{-2}\varepsilon+\min(1,\lambda)}\NN{\F_\varepsilon(u_{n+1}^{(t,h)})}_{X',\varepsilon}.
\end{align*}
Incidentally, this bound can be estimated further by means of~\eqref{upperbound}. Conversely, we notice that~\eqref{eq:up} leads to
\begin{align*}
\NN{\F_\varepsilon(u_{n+1}^{(t,h)})}_{X',\varepsilon}
&=\sup_{0\not\equiv v\in X}\frac{\langle\F_\varepsilon(u_{n+1}^{(t,h)})-\F_\varepsilon(u),v\rangle}{\NN{v}_{\varepsilon,\Omega}}
\le\max(1,L)\NN{u-u_{n+1}^{(t,h)}}_{\varepsilon,\Omega}.
\end{align*}
This gives the equivalence of the residual and the error norm~$\NN{u-u_{n+1}^{(t,h)}}_{\varepsilon,\Omega}$.
\end{remark}

%%%%%%%%%%%%%%%%%%%%%%%%%%%%%%%%%%%%%%%%%%%%%%%%%%%%%%%%%%%%%%%%%%%%%%%%%%%%%%%%%%%%%%%%%%%%%%%%%%%%%%%%%%%%%%%%%%%%%%%%%%%%%%%%%%%%%%%%%%%%%%%%%%%%%%%%%%%%%%%%%%%%%%%%%%%%

\subsubsection{Lower Bounds}

Let us sketch how $\varepsilon$-robust local lower error bounds can be derived. To this end, consider $E=\mathcal{E}_h(T)\cap\mathcal{E}_h(T')\in\mathcal{E}_h$, with $ T,T' \in \mathcal{T}_{h} $. Then, elementwise integration by parts on~$w_E$ yields
\begin{equation*}
% \label{der}
\int_{E}{\jmp{\nabla u_{n+1}^{(t,h)}}}v\ds=\int_{w_{E}}{\nabla u_{n+1}^{(t,h)}\cdot\nabla v}\dx+\int_{w_{E}}{v\Delta u_{n+1}^{(t,h)}}\dx \quad \forall v \in H^1_0(w_{E}).
\end{equation*}
Therefore, for all~$v\in H^1_0(w_{E})$, we obtain
\begin{equation*}
% \label{decomposition}
\begin{aligned}
 \int_{E}&{\varepsilon \jmp{\nabla u_{n+1}^{(t,h)}}}v\ds \\
 &= \left \langle \F_{\varepsilon}(u_{n+1}^{(t,h)}),v \right \rangle+\int_{w_E}{\{f(u_{n+1}^{(t,h)})-f^{t}(u_{n+1}^h)\}v}\dx\\
&\quad+\int_{w_E}{\{f^{t}(u_{n+1}^h)-\hat{f}^{t}(u_{n+1}^h)\}v}\dx+\int_{w_E}{\{\hat{f}^{t}(u_{n+1}^h)+\varepsilon \Delta u_{n+1}^{(t,h)}\}v}\dx,
\end{aligned}
\end{equation*}
where~$\hat{f}^{t}(u_{n+1}^h)\in V^h_0$ is the~$L^2$-projection of~$f^{t}(u_{n+1}^h)$ onto~$V_0^h$. Especially, for~$v\in H^1_0(T)$, where~$T\in\mathcal{T}_h$, this implies that
\begin{equation*}
% \label{decomposition}
\begin{aligned}
-\int_{T}{\{\hat{f}^{t}(u_{n+1}^h)+\varepsilon \Delta u_{n+1}^{(t,h)}\}v}\dx
 &= \left \langle \F_{\varepsilon}(u_{n+1}^{(t,h)}),v \right \rangle+\int_{T}{\{f(u_{n+1}^{(t,h)})-f^{t}(u_{n+1}^h)\}v}\dx\\
&\quad+\int_{T}{\{f^{t}(u_{n+1}^h)-\hat{f}^{t}(u_{n+1}^h)\}v}\dx.
\end{aligned}
\end{equation*}

Then, proceeding along the lines of~\cite{Verfuerth} by using suitable bubble function techniques, the following bounds can be proved:
\begin{equation*}
 %\label{b}
\alpha_{T}\norm{\hat{f}^{t}(u_{n+1}^h)+\varepsilon \Delta u_{n+1}^{(t,h)}}_{0,T}  \preccurlyeq  \NN{\F_{\varepsilon}(u_{n+1}^{(t,h)})}_{\varepsilon,H^1_0(T)'} + \alpha_{T}(\delta_{n,T}+\hat{\delta}_{n,T}),
\end{equation*}
and
\begin{equation*}
%\label{g}
\begin{aligned}
 \varepsilon^{-\nicefrac14} \alpha_{E}^{\nicefrac12}&\norm{\varepsilon \jmp{\nabla u_{n+1}^{(t,h)}}}_{0,E}\preccurlyeq  \NN{\F_{\varepsilon}(u_{n+1}^{(t,h)})}_{\varepsilon, H^1_0(w_E)'} +\sum_{T\in w_{E}}{\alpha_{T}(\delta_{n,T}+\hat{\delta}_{n,T})},
\end{aligned}
\end{equation*}
where, for a subset~$D\subseteq\Omega$, we let
\[
\NN{\F_\varepsilon(u)}_{\varepsilon,H^1_0(D)'} =  \sup_{\genfrac{}{}{0pt}{}{v\in H^1_0(D)}{\NN{v}_{\varepsilon,D}=1}}\int_{D}\left\{\varepsilon \nabla u\cdot \nabla v-f(u)v\right\}\dx,\qquad u\in X.
\]
Here, for~$T\in\mathcal{T}_h$,
\[
\hat\delta_{n,T}:=\norm{f^{t}(u_{n+1}^h)-\hat f^t(u_{n+1}^{h})}_{0,T}
\]
is a data oscillation term. Moreover, $\delta_{n,T}$ was introduced in~\eqref{Newtonerror}, and~$\alpha_T$ and~$\alpha_E$ were defined in~\eqref{eq:alphaT} and~\eqref{boundary}, respectively.

\subsection{A Fully Adaptive Newton-Galerkin Algorithm}
We will now propose a procedure that will combine the adaptive Newton methods presented in~Section~\ref{sc:anewton} with automatic finite element mesh refinements based on the {\em a posteriori} error estimate from Theorem~\ref{thm:1}. To this end, we make the assumption that the Newton-Raphson sequence $\left\{u_{n+1}^{(k_{n},h)}\right\}_{n\ge 0} $ given by~\eqref{eq:fem} and~\eqref{eq:ut}, with a step size~$t=k_n$, is well-defined as long as the iterations are being performed. 

\begin{algorithm}\label{al:full}
Given a parameter~$\theta>0$, a (coarse) starting mesh $ \mathcal{T}_{h} $ in~$\Omega$, and an initial guess $ u_{0}^{h} \in  V_{0}^{h} $. Set~$n:=0$.
\begin{enumerate}
\item
Determine the Newton-Raphson step size parameter $k_n$ based on~$u_n^h$ by one of the adaptive procedures from Section~\ref{sc:anewton}. 
\item Compute the FEM solution~$u_{n+1}^h$ from~\eqref{eq:fem} with step size~$t=k_n$ on the mesh~$\mathcal{T}_h$. Furthermore, obtain~$ u_{n+1}^{(k_{n},h)} $ in~\eqref{eq:ut}, and evaluate the corresponding error indicators $ \eta_{T,n} $, $ T \in \mathcal{T}_{h} $, and $\delta_{n,\Omega} $ from~\eqref{Femerror} and~\eqref{Newtonerror}, respectively.  
\item 
If there holds
\begin{equation}\label{eq:test}
\delta_{n,\Omega}^2\le \theta\sum_{T\in \mathcal{T}_{h}}{\eta_{T,n}^2},
\end{equation}
then refine the mesh $ T \in \mathcal{T}_{h} $ adaptively based on the elementwise error indicators~$\eta_{n,T}$, $T\in\mathcal{T}_h$ from Theorem~\ref{thm:1}; repeat step~(2) with the new mesh~$\mathcal{T}_h$. Otherwise, i.e. if~\eqref{eq:test} is not fulfilled, then set~$n\leftarrow n+1$, and perform another adaptive Newton step by going back to~(1).
\end{enumerate}
\end{algorithm}

\subsection{Numerical Experiments}

We will now illustrate and test the above Algorithm~\ref{al:full} with a number of numerical experiments in 1d and 2d.

\subsubsection{Problems in 1d}

In the following 1d-experiments we shall employ the fully adaptive procedure from Algorithm~\ref{al:full}, based on the improved prediction strategy from~Algorithm~\ref{al:zs} (with~$\gamma=0.5$).

\begin{example}\label{ex:1}
Let us first consider a {\em linear} singularly perturbed problem: 
\begin{equation}
\label{1.2d}
\begin{aligned}
-\varepsilon u'' + u&=1 \ \text{on } (0,1),\qquad
u_{\varepsilon}(0)= u_{\varepsilon}(1)=0. 
\end{aligned}
\end{equation}
In this case the Newton-Raphson iteration is redundant as it converges to the unique solution in one single step. Our goal is here to test the robustness of the {\em a posteriori} error analysis with respect to~$\varepsilon$ as~$\varepsilon\to 0$. 

Note that the exact solution~$u_\varepsilon$ exhibits two boundary layers at $ x  \in \{0,1\}$; see Figure ~\ref{ra} (left).
We test our algorithm by comparing the true error $ \NN{u_{h}-u}_{\varepsilon,\Omega} $ (cf.~Remark~\ref{rm:4.5}) with the estimated error (i.e., the right-hand side of~\eqref{upperbound}), and compute the efficiency indices (defined by the ratio of the estimated and true errors); the results are displayed in Figure~\ref{bild2} for $ \varepsilon = 10^{-n}$, with $n \in \{0,1,2,3,4,5\} $. For~$\varepsilon=10^{-5}$ we observe from Figure~\ref{ra} (right) that the convergence is of first order as expected. Furthermore, Figure~\ref{bild2} clearly highlights the robustness of the efficiency indices with respect to~$\varepsilon\to 0$. Here, we have used $ \theta = 0.5 $~in \eqref{eq:test}. 
\begin{figure}
\includegraphics[width=0.45\textwidth]{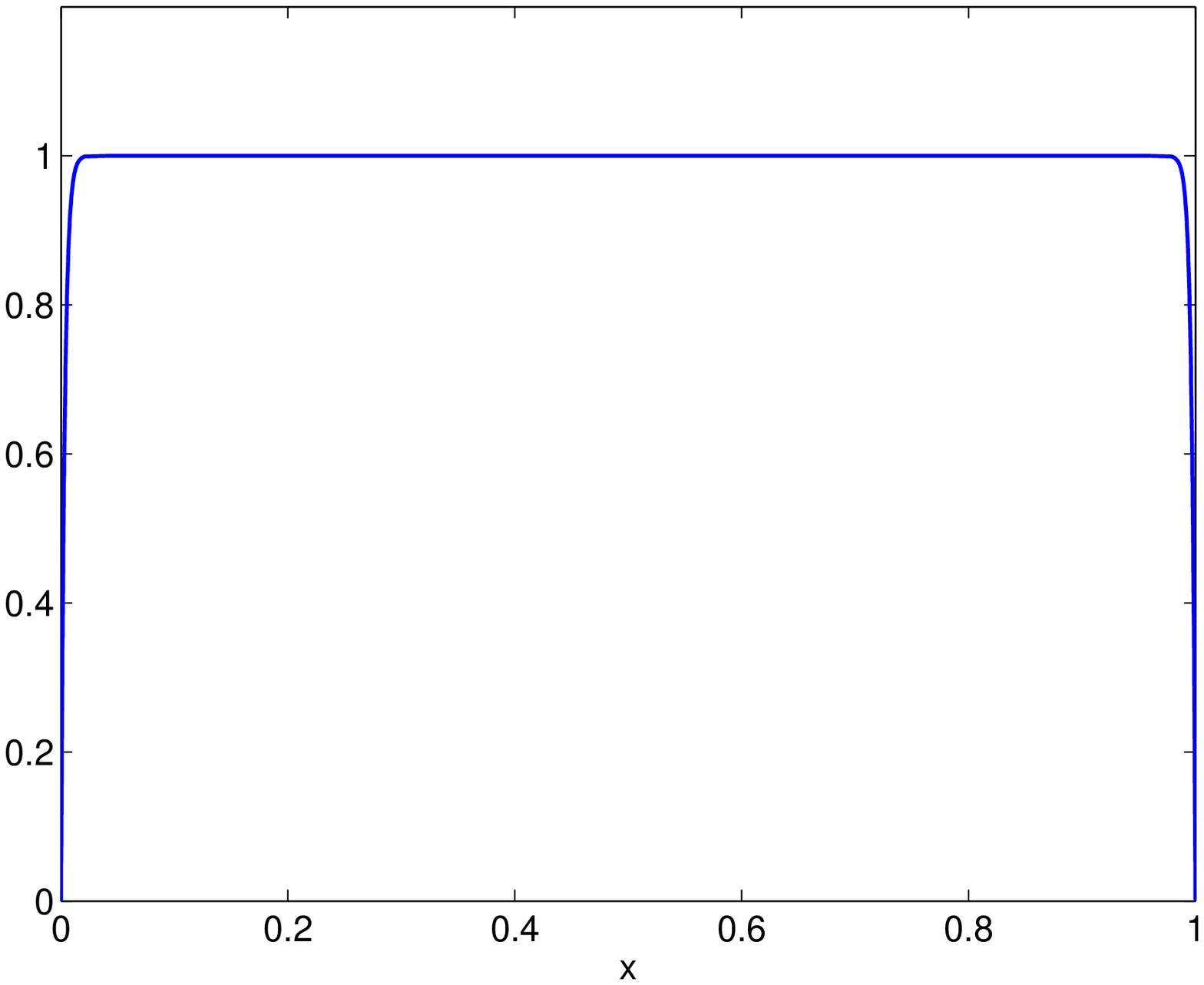}
\hfill
\includegraphics[width=0.5\textwidth]{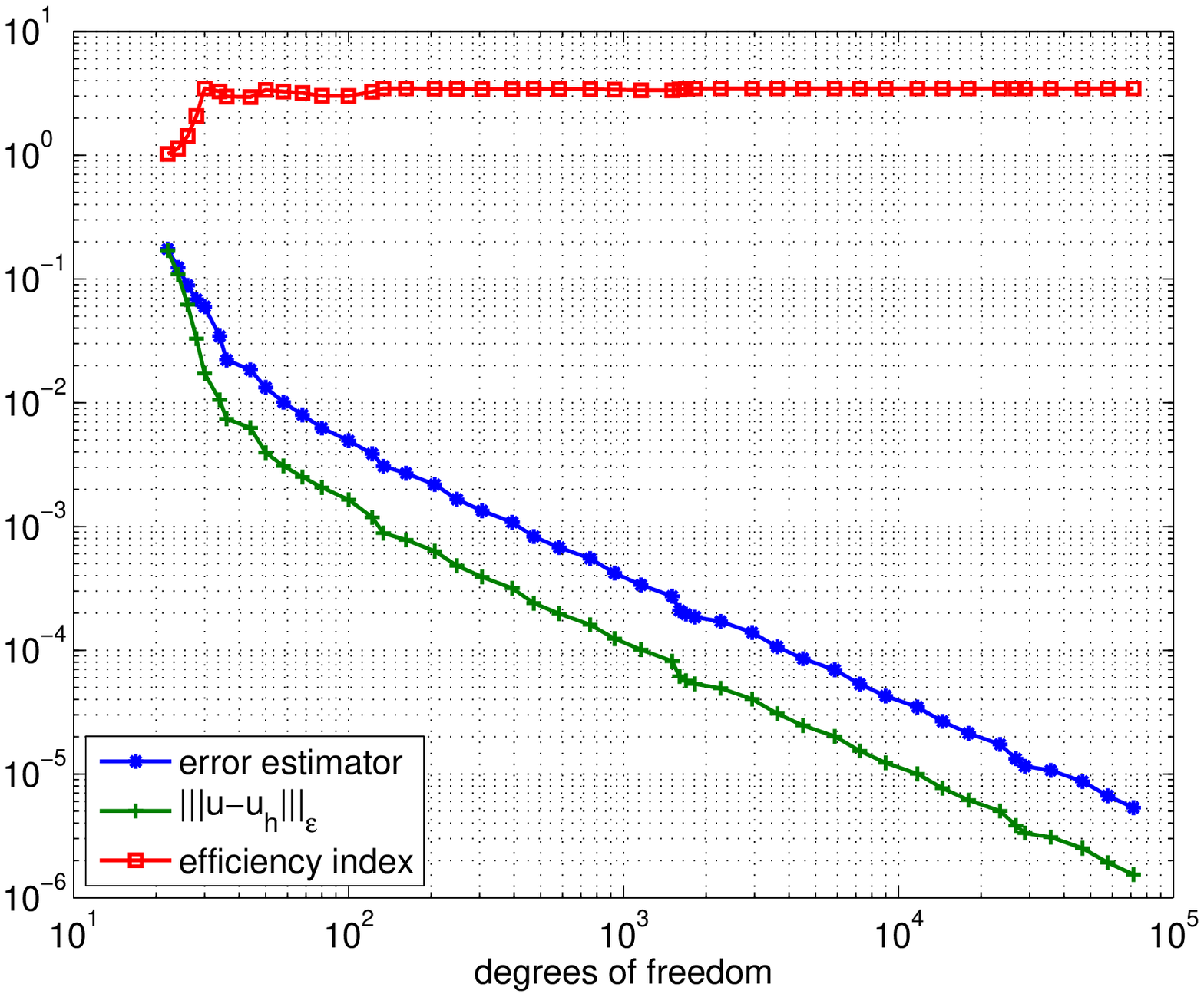}
\caption{Example~\ref{ex:1}: Exact solution (left) and performance of Algorithm~\ref{al:full} (right) for $\varepsilon=10^{-5} $.}
\label{ra}
\end{figure}
\begin{figure}
\includegraphics[width=0.48\textwidth]{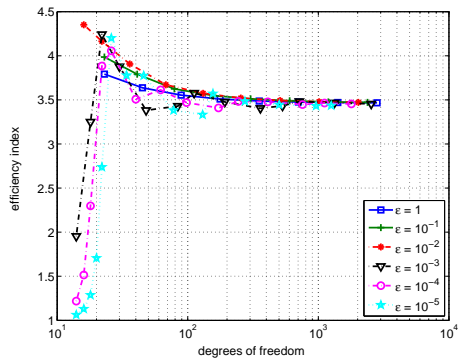}
\caption{Example~\ref{ex:1}: Efficiency indices.}
\label{bild2}
\end{figure}
\end{example}

\begin{example}\label{ex:2}
Furthermore, 
consider Fisher's equation,
\begin{equation}
\label{1.5d}
\begin{aligned}
\varepsilon u'' +u-u^2 &=0 \ \text{on} \ (0,1),\qquad
u_{\varepsilon}(0)=\alpha,\quad \ u_{\varepsilon}(1)=\beta. 
\end{aligned}
\end{equation}
A first integral form for \eqref{1.5d} is given by $E(x,y):=\varepsilon y^2-\frac{2}{3}x^3+x^2 $, from which we readily infer that the solutions have boundary layers close to $ x=0 $ and $ x=1 $. Furthermore, for $ \alpha>-\nicefrac12 $ and  $\beta < 1 $, the solutions feature an increasing number of spikes (which are bounded by~1) as $\varepsilon \to 0$ (see Figure~\ref{spikes}). There are infinitely many solutions (for which there are no analytical solution formulas available in general); see, e.g., \cite{Verhulst} for a more detailed discussion.

In our example, we have started the Newton-Raphson iteration based on a uniform grid with~$100$ nodes, and an initial spike-like function depicted on the left in Figure~\ref{spikes}. Again, we set $\theta = 0.5 $~in \eqref{eq:test}, and perform our experiments for $\tau = 0.1 $ in Algorithm~\ref{al:simple}, and $\varepsilon=0.00025$. 

In Figure~\ref{bild6} we depict the performance of the error estimator. The fully adaptive Newton-Galerkin scheme converges to a numerical solution as shown on the right of Figure~\ref{spikes}. We emphasize that our scheme is able to transport the initial function to a numerical solution which is of similar shape; in particular, it seems clear that the iteration has remained  in the attractor of the solution which contains the initial guess. It is well-known that this will typically not happen for the traditional Newton scheme (with fixed step size~1), or even for a damped Newton method (with fixed step size smaller than~1); indeed, for this type of problem with~$\varepsilon\ll 1$, these methods will mostly fail to converge to a bounded solution at all (see, e.g., \cite{ChNiZo00}). 
\begin{figure}
\includegraphics[width=0.45\textwidth]{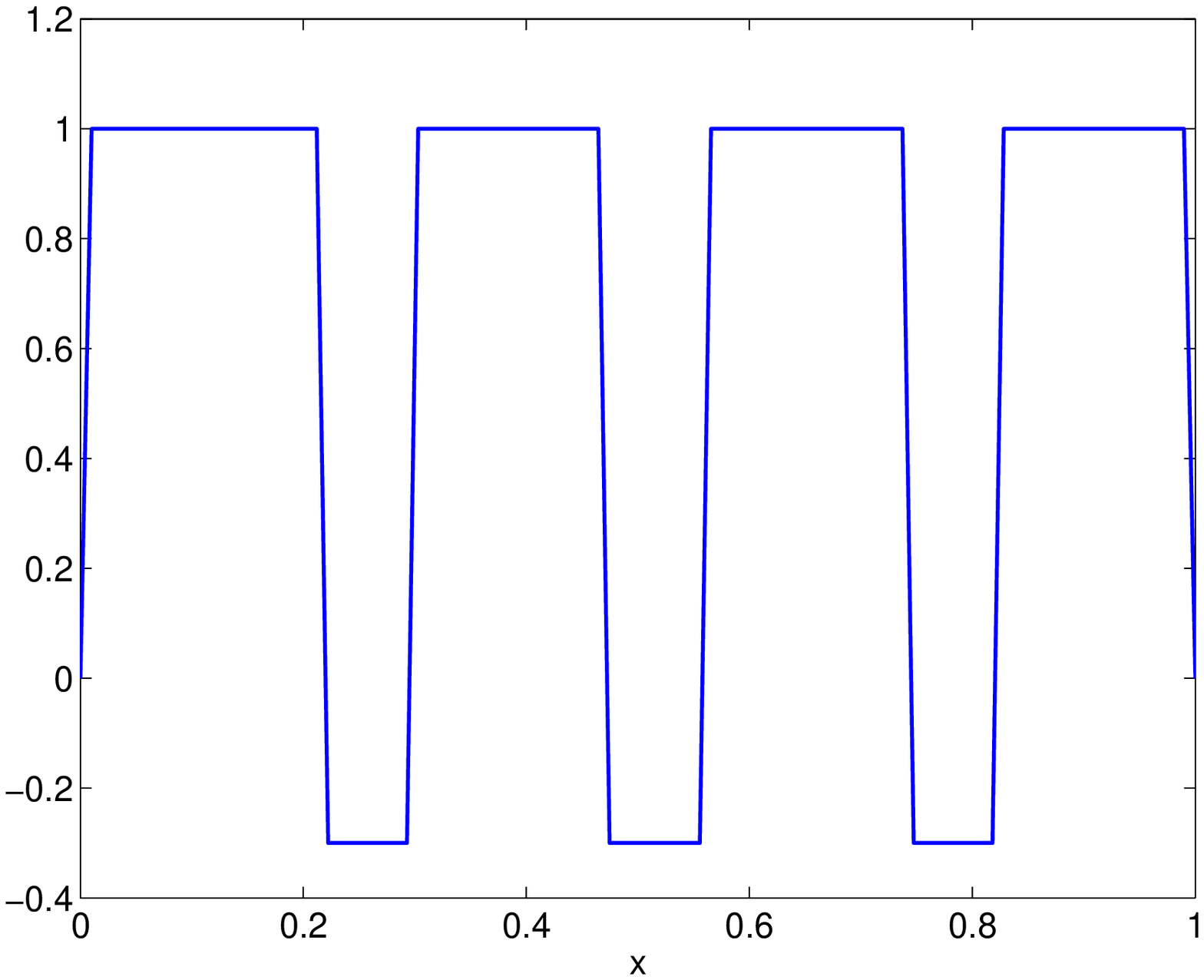}
\hfill
\includegraphics[width=0.45\textwidth]{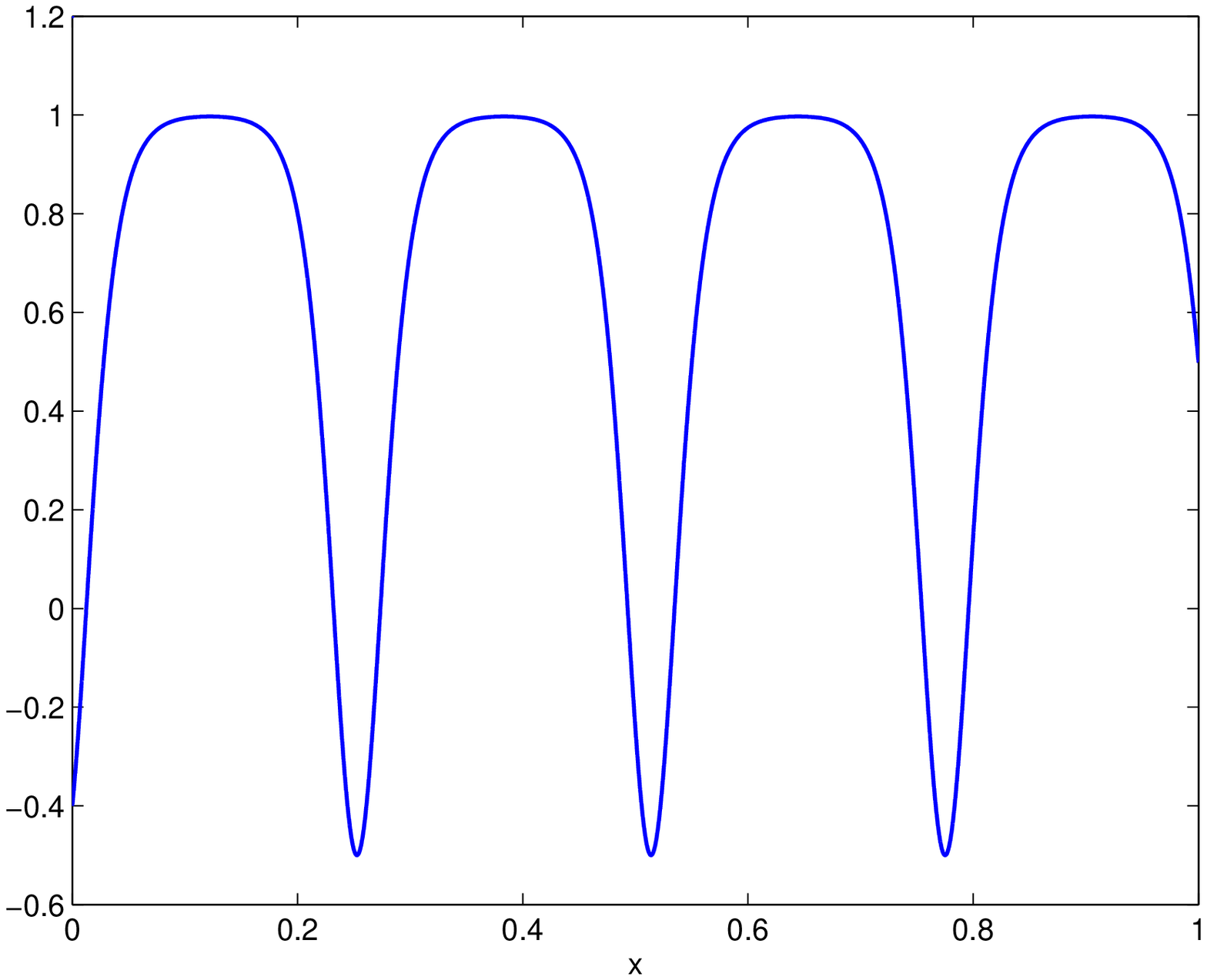}
\caption{Example~\ref{ex:2}: Initial data (left) and numerical solution resulting from Algorithm~\ref{al:full} (right) with $\alpha=-0.4 $, $ \beta=0.5 $, and $ \varepsilon=0.00025$.}
\label{spikes}
\end{figure}
\begin{figure}
\includegraphics[width=0.48\textwidth]{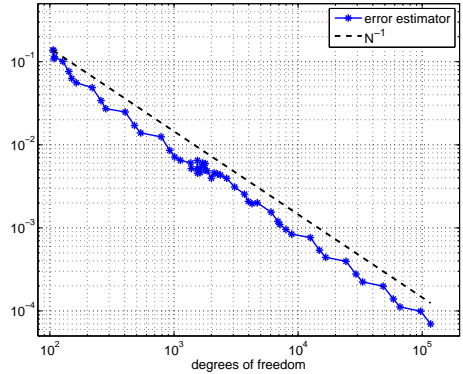}
\caption{Example~\ref{ex:2}: Estimated error for $ \varepsilon=0.00025$.}
\label{bild6}
\end{figure}
\end{example}

\subsubsection{A Problem in 2d}

We will now turn to a 2d-example, where we shall employ the simple prediction strategy presented in Algorithm \ref{al:simple} (see also~\cite{AmWi14}) for the selection of the local Newton-Raphson step size.

\begin{example}\label{ex:5}
Consider the well-known nonlinear Ginzburg-Landau equation on the unit square~$\Omega=(0,1)^2$ given by
\begin{equation}
\label{32d}
\begin{aligned}
\varepsilon\Delta u - u^3+u&=0 \ \text{in } \Omega,\qquad
u = 0 \ \text{on } \partial \Omega.\\
\end{aligned}
\end{equation}
Clearly $ u\equiv 0 $ is a solution. In addition, any solution~$u$ appears pairwise as $-u$ is obviously a solution also. Neglecting the boundary conditions for a moment, one observes that $u\equiv1$ and $u \equiv-1$ are solutions of the PDE. We therefore expect boundary layers along $ \partial \Omega$, and possibly within the domain~$\Omega$; see Figure ~\ref{bild11}, where we depict two different solutions of problem~\eqref{32d}. 

The solution on the top left in Figure~\ref{bild11} was obtained from choosing the initial function $(x_1,x_2)\mapsto\sign(x_2) $, 
whereas the solution on the bottom left was computed by choosing $(x_1,x_2)\mapsto -1 $ (both with enforced zero Dirichlet boundary conditions at the boundary degrees of freedom). The perturbation parameter is chosen to be~$\varepsilon=0.5\cdot 10^{-5}$. We restrict the Newton step size in Algorithm \ref{al:simple} by choosing $ \tau=0.1$. Moreover we have set $ \theta=0.75 $. Again the performance data illustrated on the right-hand side in Figure~\ref{bild11} indicates (optimal) first-order convergence as expected.
\begin{figure}
\includegraphics[width=0.45\textwidth]{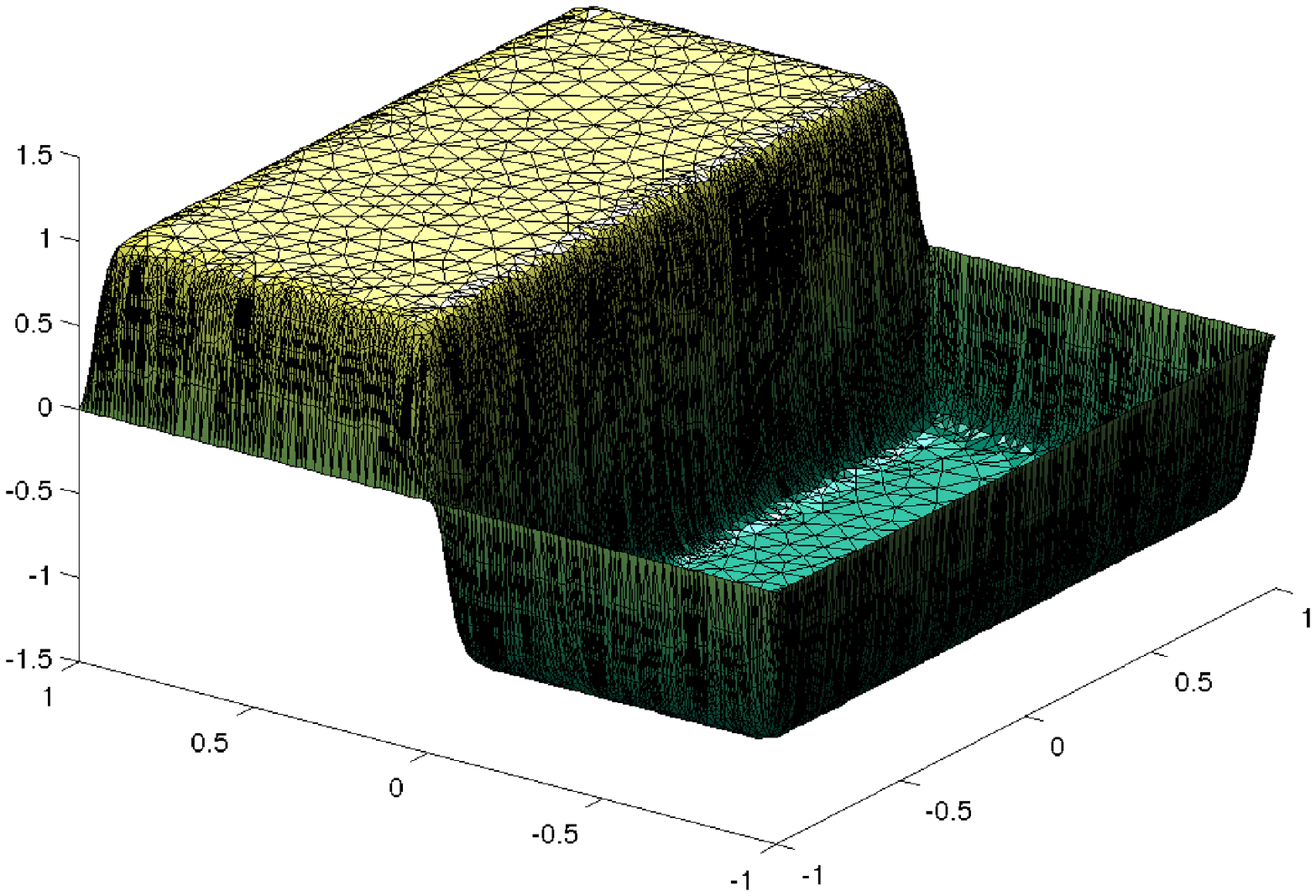}
\hfill
\includegraphics[width=0.45\textwidth]{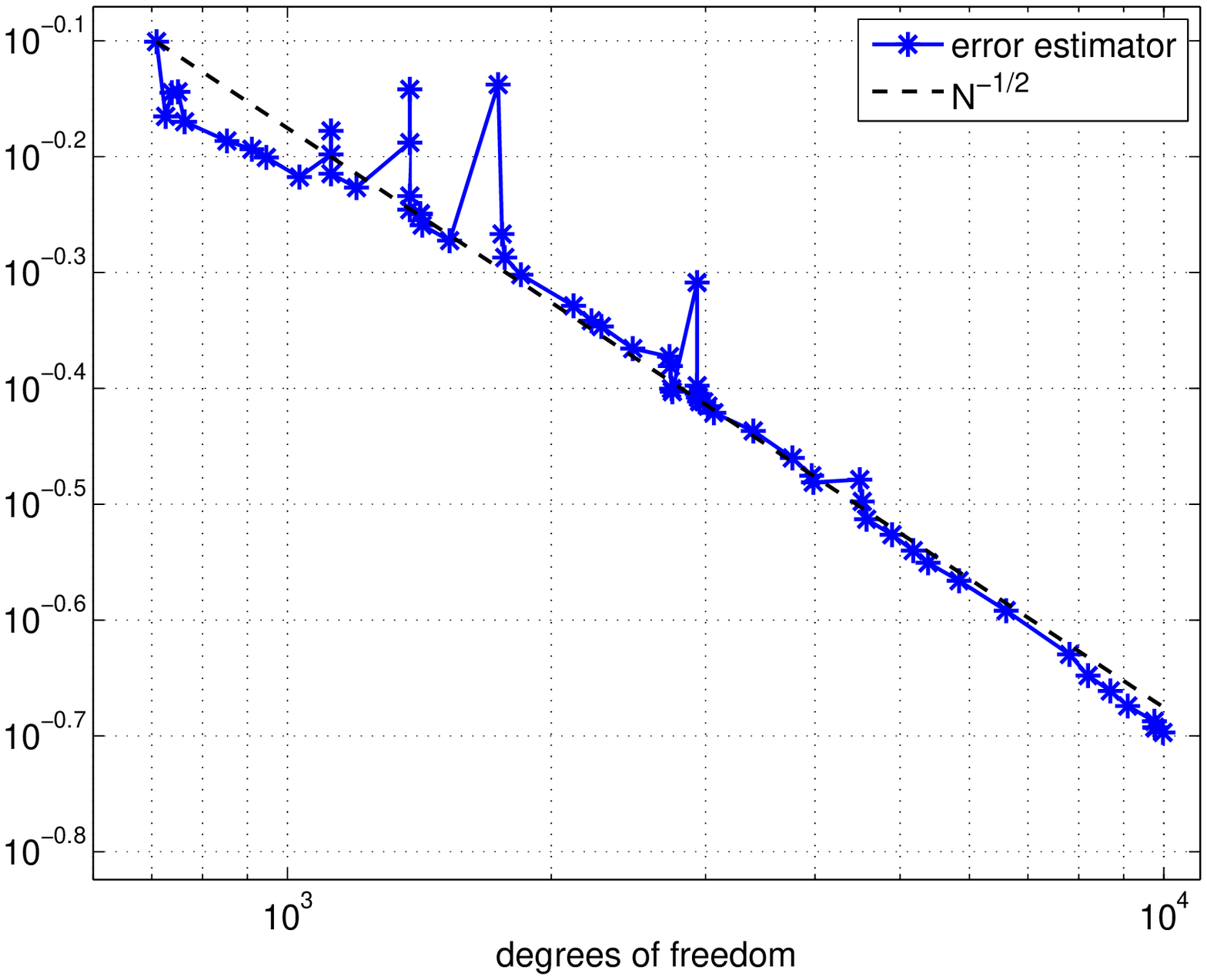}
\hfill
\includegraphics[width=0.45\textwidth]{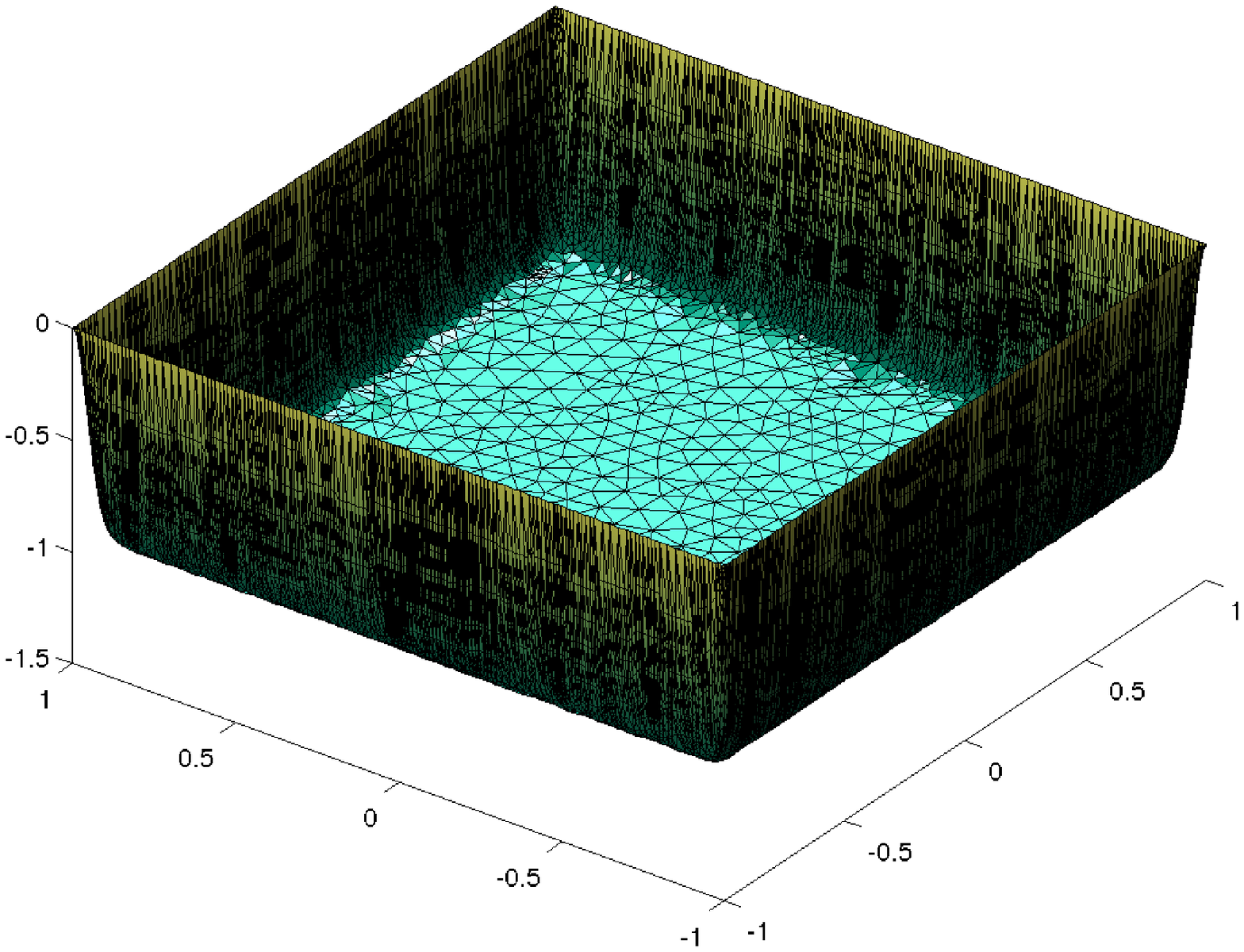}
\hfill
\includegraphics[width=0.45\textwidth]{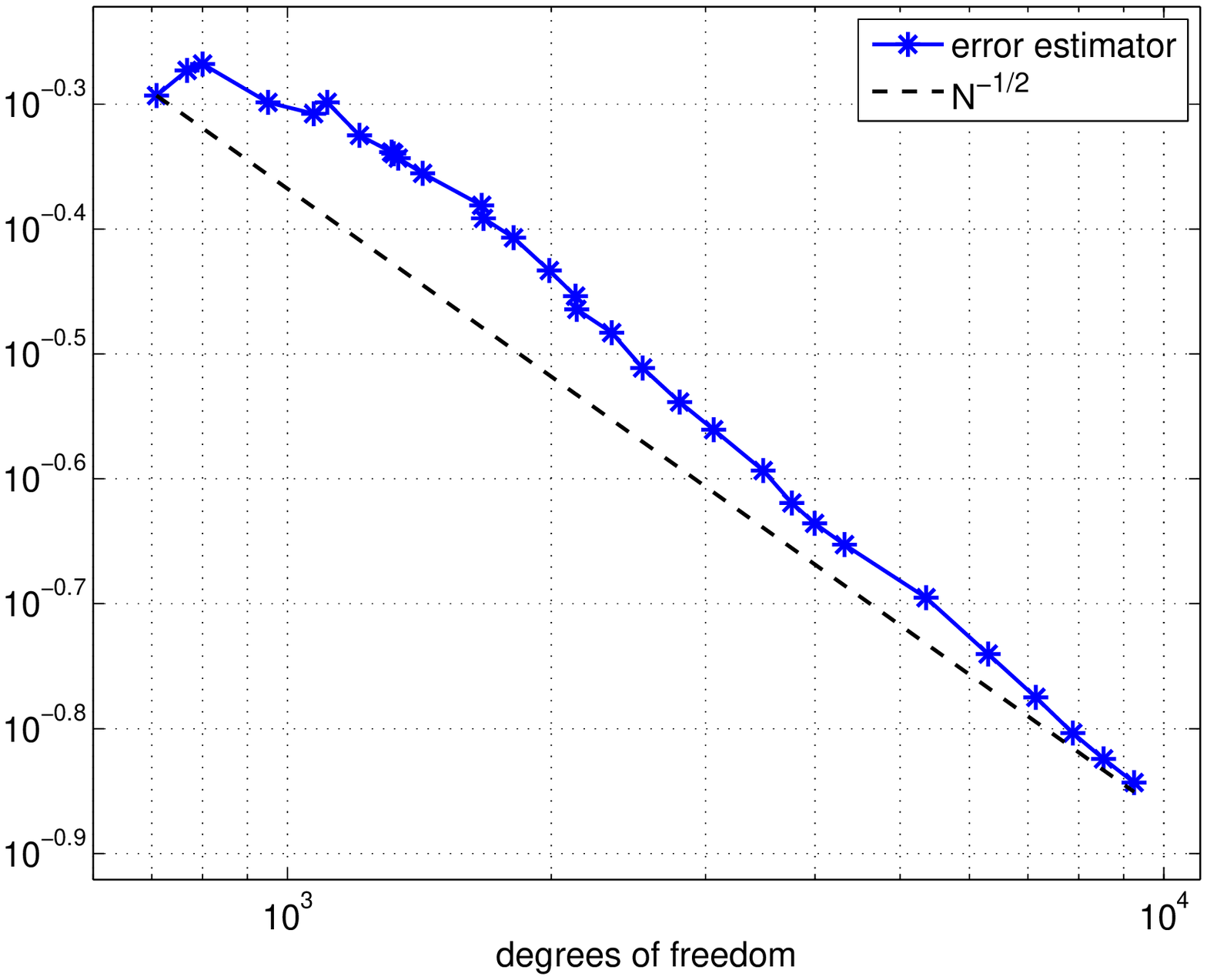}
\caption{Example~\ref{ex:5}: Numerical solutions (left) and the corresponding performances (right) with $\varepsilon=0.5\cdot 10^{-5} $ for different initial guesses.}
\label{bild11}
\end{figure}
\end{example}

\section{Conclusions}\label{sc:concl}
The aim of this paper was to introduce a reliable and computationally feasible procedure for the numerical solution of general, semilinear elliptic boundary value problems with possible singular perturbations. The key idea is to combine an adaptive Newton-Raphson method with an automatic mesh refinement finite element procedure. Here, the (local) Newton-Raphson damping parameter is selected based on interpreting the scheme within the context of step size control for dynamical systems. Furthermore, the sequence of linear problems resulting from the Newton discretization is treated by means of a robust (with respect to the singular perturbations) {\em a posteriori} residual-oriented error analysis and a corresponding adaptive mesh refinement scheme. Our numerical experiments clearly illustrate the ability of our approach to reliably find solutions reasonably close to the initial guesses, and to robustly resolve the singular perturbations at an optimal rate.

\appendix
\section{A Sobolev Inequality}

\begin{lemma}\label{lm:app}
Let~$\Omega\subset\mathbb{R}^d$ be a bounded open interval ($d=1$), or a bounded Lipschitz domain ($d=2$). Then, if~$g\in L^{1+\beta(d-1)}(\Omega)$, for some~$\beta\in(0,1]$, then there holds that
\[
\|guv\|_{L^1(\Omega)}\preccurlyeq\|g\|_{L^{1+\beta(d-1)}(\Omega)}\|\nabla u\|_{L^2(\Omega)}\|\nabla v\|_{L^2(\Omega)},
\]
for any~$u,v\in H^1_0(\Omega)$.
\end{lemma}

\begin{proof} We treat the cases~$d=1$ and~$d=2$ separately.

\paragraph{\em Case $d=1$:}
By the Sobolev embedding theorem and the Poincar\'e inequality there holds $ \norm{u}_{L^{\infty}(\Omega)}\preccurlyeq  \norm{\nabla{u}}_{L^2(\Omega)} $. 
Thence, we get
\begin{equation}\label{eq:1d}
\norm{guv}_{L^{1}(\Omega)} \le \norm{g}_{L^{1}(\Omega)}\norm{uv}_{L^{\infty}(\Omega)}\preccurlyeq \norm{g}_{L^{1}(\Omega)}\norm{\nabla(uv)}_{L^{2}(\Omega)}.
\end{equation}
Furthermore, due to the product rule and the triangle inequality, we have
\begin{equation}
\begin{split}
\label{viceversa}
\norm{\nabla(uv)}_{L^{2}(\Omega)}
&\leq \norm{u\nabla{v}}_{L^{2}(\Omega)}+\norm{v\nabla{u}}_{L^{2}(\Omega)}\\
&\le \norm{u}_{L^{\infty}(\Omega)}\norm{\nabla{v}}_{L^{2}(\Omega)}+\norm{v}_{L^{\infty}(\Omega)}\norm{\nabla{u}}_{L^{2}(\Omega)}\\
&\preccurlyeq \norm{\nabla u}_{L^{2}(\Omega)}\norm{\nabla{v}}_{L^{2}(\Omega)}.
\end{split}
\end{equation}
Inserting this bound into~\eqref{eq:1d} completes the argument for~$d=1$.
%%%
\paragraph{\em Case $d=2$:}
We choose~$\delta\in(0,1]$ to be specified later, and set~$p(\delta):=\nicefrac{(4-2\delta)}{(4-3\delta)}\in(1,2]$ and~$q(\delta):=\nicefrac{(4-2\delta)}{\delta}\in[2,\infty)$, so that~$p(\delta)^{-1}+q(\delta)^{-1}=1$. Then, by means of H\"older's inequality, we note that
\begin{equation}
\label{app1}
\norm{g uv}_{L^{1}(\Omega)}\leq \norm{g}_{L^{p(\delta)}(\Omega)}\norm{uv}_{L^{q(\delta)}(\Omega)}.
\end{equation}
Here, referring to~\cite[Theorem~3.4.3]{jost}), there holds 
\begin{equation}
\label{sobolevineq}
\norm{uv}_{L^{q(\delta)}(\Omega)}\preccurlyeq \norm{\nabla(uv)}_{L^{r(\delta)}(\Omega)},
\end{equation}
with~$ r(\delta):=2-\delta\in[1,2)$. Using the product rule together with the triangle inequality, results in
\begin{equation}
\label{app2}
\norm{\nabla{(uv)}}_{L^{r(\delta)}(\Omega)}\leq \norm{u\nabla{v}}_{L^{r(\delta)}(\Omega)}+ \norm{v\nabla{u}}_{L^{r(\delta)}(\Omega)}.
\end{equation}
Then, invoking H\"older's inequality again as well as~\eqref{sobolevineq}, we see that
\begin{equation}
\label{app3}
\begin{aligned}
\norm{u\nabla{v}}_{L^{r(\delta)}(\Omega)}
&\le\norm{u}_{L^{\nicefrac{2r(\delta)}{(2-r(\delta))}}(\Omega)}\norm{\nabla{v}}_{L^{2}(\Omega)}\\
&=\norm{u}_{L^{q(\delta)}(\Omega)}\norm{\nabla{v}}_{L^{2}(\Omega)}
\preccurlyeq \norm{\nabla{u}}_{L^{r(\delta)}(\Omega)}\norm{\nabla{v}}_{L^{2}(\Omega)}\\
&\preccurlyeq \norm{\nabla{u}}_{L^{2}(\Omega)}\norm{\nabla{v}}_{L^{2}(\Omega)},
\end{aligned}
\end{equation}
and similarly,
\begin{equation}\label{app3b}
 \norm{v\nabla{u}}_{L^{r(\delta)}(\Omega)}
 \preccurlyeq \norm{\nabla{v}}_{L^{2}(\Omega)}\norm{\nabla{u}}_{L^{2}(\Omega)}.
\end{equation}
Combining \eqref{app1}--\eqref{app3b}, we end up with 
\[
\|guv\|_{L^1(\Omega)}\preccurlyeq\|g\|_{L^{p(\delta)}(\Omega)}\|\nabla u\|_{L^2(\Omega)}\|\nabla v\|_{L^2(\Omega)},
\]
which shows the claim with $p(\delta)=1+\beta\in(1,2]$.
\end{proof}

\bibliographystyle{amsplain}
\bibliography{references}

\end{document}